\documentclass[12pt]{article}

\AtBeginDocument{}

\usepackage{amsmath}
\usepackage{amssymb}
\usepackage{titling}
\usepackage[margin= 1in]{geometry}
\usepackage{graphicx}
\usepackage{mathtools}
\usepackage{graphicx}
\usepackage{xcolor}
\usepackage{amsthm}
\usepackage[utf8]{inputenc}
\usepackage[english]{babel}
\usepackage{dsfont}
\usepackage[all,cmtip]{xy}
\usepackage{subcaption}
\usepackage{float}
\usepackage{framed,enumitem}
\usepackage[nodisplayskipstretch]{setspace}
\usepackage{array}
\usepackage{tabularray}
\usepackage[export]{adjustbox}
\usepackage{multirow}
\usepackage{pdfpages}
\usepackage{comment}

\renewcommand{\epsilon}{\varepsilon}
\renewcommand{\phi}{\varphi}
\renewcommand{\bar}{\overline}

\renewcommand{\leq}{\leqslant}
\renewcommand{\geq}{\geqslant}

\newcommand{\qq}{\hspace{.15cm}}

\newtheorem{theorem}{Theorem}[section]

\newtheorem{lemma}[theorem]{Lemma}

\theoremstyle{definition}

\begin{document}

\title{Two-Orbit Polytopes}
 
\author{Isabel Hubard
\thanks{Email:\ isahubard@im.unam.mx
}\\
Instituto de Matem\'aticas \\
Universidad Nacional Aut\'onoma de M\'exico\\
04510 Ciudad de M\'exico, M\'exico
\and and \\[.05in]
Egon Schulte
\thanks{Email:\ e.schulte@northeastern.edu
}\\
Northeastern University\\
Boston, Massachusetts,  USA, 02115\bigskip}

\date{ \today } 
\maketitle 

\begin{abstract}

Abstract polytopes are combinatorial structures with distinctive geometric, algebraic, or topological characteristics, that generalize (the face lattice of) traditional polyhedra, polytopes or tessellations. Most research has focused on abstract polytopes with the highest possible symmetry, in particular those that are regular or chiral.
In this paper we study two-orbit polytopes, that is, abstract polytopes whose automorphism groups have exactly two orbits on flags. Such polytopes of rank $n$ fall into $2^n-1$ classes, determined by their local flag configuration.

We develop a general structural theory of two-orbit polytopes of arbitrary rank. In particular, we determine their face- and section-transitivity properties and describe the structure of their automorphism groups via distinguished generating sets and face stabilizer subgroups. These results yield a characterization of the partial order { on the polytope} in terms of the automorphism group. Two-orbit polytopes in different classes behave quite differently.

Our approach extends the group-theoretic framework for regular and chiral polytopes and provides a systematic foundation for the study of polytopes with two flag orbits.

\bigskip\medskip

\noindent
Key Words:  abstract polytope, regular polytope, chiral polytope, 2-orbit polytope.\medskip

\noindent
MSC Subject Classification (2020):\ 52B15, 51M20, 05E16, 20B25
\end{abstract}

\section{Introduction}
\label{intro}
Symmetry is a frequently recurring theme in science. 
Throughout history, the traditional regular polyhedra and polytopes have attracted much attention and inspired many new developments in mathematics and science more generally. 

Over the past few decades, the study of highly-symmetric structures has been extended in several directions centered around an abstract combinatorial polytope theory and a combinatorial notion of regularity. Abstract polytopes are combinatorial structures with distinctive geometric, algebraic, and topological properties. 
They are ranked, partially ordered sets that generalize the face lattices of convex polytopes and tessellations to a much broader class of objects. Much research on abstract polytopes focuses on the classification of regular, chiral, or other highly symmetric abstract or geometric polytopes and their automorphism or symmetry groups. The present paper  investigates abstract two-orbit polytopes, that is, abstract polytopes whose automorphism groups have exactly two orbits on the so-called flags (maximal chains of the poset).

Abstract polytopes, originally called incidence-polytopes, were introduced as a particularly polytope-like class of incidence complexes in Danzer \& Schulte~\cite{DS1982} and were inspired by Coxeter's and Gr\"unbaum's work, with deep connections to Tits' work on incidence geometries surfacing soon afterwards (see~\cite{C1973,C1991,G1977b,G2003,T1961,T1974}). The name ``abstract polytope" was adopted during the writing of the book~\cite{MS2002}. The initial focus of research was on abstract polytopes with maximum possible combinatorial symmetry, the abstract regular polytopes, which are characterized by a flag-transitive automorphism group and an abundance of ``combinatorial reflections". The comprehensive Abstract Regular Polytopes monograph~\cite{MS2002} by McMullen and Schulte is the standard reference on abstract regular polytopes and has motivated significant further research on abstract polytopes in general. 

Abstract chiral polytopes were introduced in Schulte \& Weiss~\cite{SW1991,SW1994} as a generalization of irreflexible maps on closed surfaces to higher rank structures (Coxeter \& Moser~\cite{CM1980}). The term ``chiral" was adopted at that time with Coxeter's blessing. Abstract chiral polytopes have an automorphism group with two orbits on the flags such that any two adjacent flags (differing by just one face) lie in distinct flag orbits. These polytopes have maximal symmetry by ``combinatorial rotation" but unlike abstract regular polytopes, lack sufficient symmetry by ``combinatorial reflection". Finite abstract chiral polytopes of rank $n>4$ were initially elusive, but were shown to exist in Conder, Hubard \& Pisanski~\cite{CHP2008} for rank 5 and subsequently, in Pellicer~\cite{P2010} for arbitrary rank. For example, for any rank $n>4$, both the symmetric group $S_m$ and alternating group $A_m$ occur as automorphism groups of a chiral $n$-polytope with simplex facets for all but finitely many degrees $m$ (see Conder, Hubard \& O’Reilly-Regueiro~\cite{CHO2024}). The recent Abstract Chiral Polytopes monograph by Pellicer~\cite{P2025} provides a state of the art account on abstract chiral polytopes.



While most work has focused on polytopes with maximal symmetry (either by reflections or rotations), this leaves open a natural and highly structured intermediate class: that of two-orbit polytopes.
Two-orbit polytopes are those whose automorphism groups have exactly two orbits on flags. 
These objects form the first genuinely nontrivial level of symmetry beyond the classical cases. 
Unlike chiral polytopes, where adjacency rigidly separates flag orbits, general two-orbit polytopes allow adjacency to either preserve or switch orbits, leading to a range of symmetry behaviors governed by local combinatorial configuration.

This added flexibility has strong and measurable consequences. As already pointed out in \cite{isathesis}, two-orbit polytopes of rank $n$ naturally fall into $2^n-1$ distinct classes, indexed by subsets of ranks indicating which flag adjacencies preserve orbit type. 
This classification has direct consequences for their {combinatorics and symmetry properties}. In particular, it determines their face-transitivity and section-transitivity behavior: while regular and chiral polytopes are fully transitive, two-orbit polytopes may have either one or two orbits on faces or sections, depending on the class.

Examples already arise in rank 3. The {familiar} cuboctahedron and icosidodecahedron, together with their duals, are non-chiral two-orbit polyhedra exhibiting different local configurations while retaining significant symmetry.
In general, two-orbit polytopes need not be equivelar, and their local structure is naturally described by double Schläfli symbols, reflecting the presence of two distinct types of sections {of rank 2 at certain levels in the face poset.}

From the algebraic perspective, two-orbit polytopes admit a robust and uniform group-theoretic description. Their automorphism groups possess distinguished generating systems that extend those of regular and chiral polytopes, enabling a systematic analysis of generators, stabilizers, and the underlying combinatorial structure. 
This interplay between local combinatorics and global symmetry makes two-orbit polytopes particularly amenable to structural classification.

The main purpose of this paper is to develop a general theory of two-orbit polytopes of arbitrary rank. We describe their classification into classes, determine their face- and section-transitivity properties, and establish structural results for their automorphism groups, including explicit generating systems and stabilizer subgroups. Our results show that the behavior of two-orbit polytopes depends strongly on their class, with qualitatively different phenomena arising in different cases, and provide a foundation for further study of polytopes with few flag orbits.


Our starting point is the investigation of abstract two-orbit polyhedra in Hubard~\cite{H2010}. 
The present article has existed in nearly complete form as an unpublished preprint for several years. Our results have inspired and influenced many new developments on symmetric structures with multiple flag orbits, and the paper has been cited repeatedly in the literature, attesting to the significance of our results. We further elaborate on this aspect at the end of this paper.

The article is organized as follows. In Section~\ref{bano}, we review the concept of an abstract polytope and remind the reader of basic results about the automorphism groups of regular or chiral polytopes. Section~\ref{twoorbitspols} explains how two-orbit polytopes of rank $n$ naturally fall into ${2^{n}-1}$ classes $2_I$, each indexed by a proper subset $I$ of the rank set $N:=\{0,\ldots,n-1\}$ called the respective class type set. The combinatorics of two-orbit polytopes heavily depends on their class type set $I$. The chiral polytopes, as two-orbit polytopes with $I=\emptyset$, represent one end of the spectrum. The other end is occupied by the regular polytopes, which have one flag orbit and correspond to the excluded case $I=N$. Section~\ref{twoorbitspols} also determines the face-transitivity and section-transitivity properties of two-orbit polytopes. In Section~\ref{grour}, we establish structure results for the automorphism groups of two-orbit polytopes following the blueprint for similar approaches for regular and chiral polytopes. In particular, distinguished generators for the automorphism group are discovered, face stabilizer subgroups are determined, and the partial order on the two-orbit polytope is characterized in terms of the generators and distinguished subgroups. Finally, Section~\ref{srd} deals with special classes of two-orbit polytopes featuring small reflection deficiency.
\smallskip

\section{Abstract polytopes} 
\label{bano}

In this section we briefly review basic notions and results about abstract polytopes. For more details the reader is referred to McMullen \& Schulte~\cite[Chs. 2,\,3]{MS2002} and Pellicer~\cite{P2025}. 
\smallskip

An (\emph{abstract\/}) \emph{polytope of rank\/} $n$, or simply an \emph{$n$-polytope\/}, is a partially ordered set $\mathcal{P}$ with a strictly monotone rank function with range $\{-1,0, \ldots, n\}$. An element of rank~$j$ is called a \emph{$j$-face\/} of $\mathcal{P}$, and a face of rank $0$, $1$ or $n-1$ is also called a \emph{vertex\/}, \emph{edge\/} or \emph{facet\/}, respectively. A {\em chain of $\mathcal{P}$} is a totally ordered subset of $\mathcal{P}$.  The maximal chains, or \emph{flags}, all contain exactly $n + 2$ faces, including a unique least face $F_{-1}$ (of rank $-1$) and a unique greatest face $F_n$ (of rank $n$). The faces $F_{-1}$ and $F_n$ are said to be the \emph{improper} faces of~$\mathcal{P}$; all other faces are \emph{proper} faces of $\mathcal{P}$. A polytope $\mathcal{P}$ also satisfies the following homogeneity condition known as the {\em diamond condition}:\ whenever $F \leq G$, with $F$ a $(j-1)$-face and $G$ a $(j+1)$-face for some~$j=0,\ldots,n-1$, there are exactly two $j$-faces $H$ with $F \leq H \leq G$. Two flags are said to be \emph{adjacent} if they differ in a single face, or $i$-\emph{adjacent})  if they differ in just their $i$-face. The diamond condition, rephrased,  is saying that every flag $\Phi$ of $\mathcal{P}$ has a unique $i$-adjacent flag, denoted~$\Phi^i$, for each $i=0, \dots, n-1$. Finally, $\mathcal{P}$ is \emph{strongly flag-connected}, in the sense that any two flags $\Phi$ and $\Psi$ of $\mathcal{P}$ can be joined by a finite sequence of successively adjacent flags, each containing $\Phi \cap \Psi$. An abstract polytope of rank 3 is also called an \emph{abstract polyhedron}, or simply \emph{polyhedron}.

In designating flags of an $n$-polytope $\mathcal{P}$, we usually suppress their improper faces, the least face $F_{-1}$ and the largest face $F_n$. For a flag $\Phi$ of $\mathcal{P}$ and integers $i_1,\ldots,i_{l}$ with $l\geq 2$ and $0\leq i_1,\ldots,i_l\leq n-1$, we inductively define a new flag via adjacency,  
\[  \Phi^{i_{1},\ldots,i_{l}}:= (\Phi^{i_{1},\ldots,i_{l-1}})^{i_{l}},\]
using $i_1,\ldots,i_l$ as superscripts. Then, by definition, any two successive flags in a flag sequence of the form 
\[\Phi,\Phi^{i_1},\Phi^{i_{1},i_{2}}, \ldots, \Phi^{i_{1},\ldots,i_{l}}\]  
are adjacent. Note that $\Phi^{i,i} = \Phi$ for each $i$, and that $\Phi^{i,j}=\Phi^{j,i}$ whenever $|i-j|>1$. We sometimes omit the commas between the superscripts and simply write $\Phi^{i_{1}\ldots i_{l}}$.

We also use integers as subscripts on flags, this time to designate faces of a given flag. For a flag $\Phi$ of $\mathcal{P}$ and for $i=-1,0,\dots,n$, we let $\Phi_i$ denote its $i$-face. Thus $\Phi_i$ is a face but $\Phi^i$ is a flag. Then notice that 
\[\Phi = \{\Phi_0,\ldots,\Phi_{n-1}\},\]
where the improper faces of $\Phi$ were suppressed. We also set $\Phi^j_i := (\Phi^j)_i$ for all $i$ and $j$, and recall that this is the $i$-face of the $j$-adjacent flag of $\Phi$. Note that $\Phi^j_i = \Phi_i$ if $j \neq i$, but that $\Phi^i_i \neq \Phi_i$. 

Occasionally we also use subscripts to label successive flags of a sequence. In this case the interpretation of the subscripts should be clear from the context.

For any two faces $F$ of rank $j$ and $G$ of rank $k$ with $F \leq G$, we call
\[G/F := \{ H \in \mathcal{P}\, | \, F \leq H \leq G \}\] 
a \emph{section} of $\mathcal{P}$. This is a $(k-j-1)$-polytope in its own right, and we sometimes refer to it as a $(k-j-1)$-section of $\mathcal{P}$. In particular, we can identify a face $F$ with the section $F/F_{-1}$. Moreover, $F_{n}/F$ is said to be the \emph{co-face of $\mathcal{P}$ at\/} $F$, or the \emph{vertex-figure of $\mathcal{P}$ at\/} $F$ if $F$ is a vertex. Note that if $\Phi$ is a flag and $j \leq k$, then the section of~$\mathcal{P}$ determined by the $j$-face and $k$-face of $\Phi$ is given by $\Phi_k / \Phi_j$. 

An {\em automorphism} of a polytope $\mathcal{P}$ is an order preserving bijection of $\mathcal{P}$ with an order preserving inverse. For a polytope $\mathcal{P}$ we let $\Gamma(\mathcal{P})$ denote its \textit{automorphism group\/}. Each automorphism of $\mathcal{P}$ induces a bijection of the set of flags $\mathcal{F}(\mathcal{P})$ of $\mathcal{P}$ that preserves flag adjacencies. More precisely, if $\gamma$ is an automorphism of $\mathcal{P}$ and $\Phi$ is a flag of $\mathcal{P}$, then
\begin{eqnarray}
\label{auto-adjacent}
(\Phi^i)\gamma = (\Phi\gamma)^i, \;\; \mathrm{for}\ i=0,1,\dots, n-1.
\end{eqnarray}
The flag-connectedness of $\mathcal{P}$ implies that $\Gamma(\mathcal{P})$ acts freely (or semi-regularly) on $\mathcal{F}(\mathcal{P})$. 
\smallskip

A polytope $\mathcal{P}$ is said to be \emph{regular\/} if $\Gamma(\mathcal{P})$ acts  transitively on the flags. In this case $\Gamma(\mathcal{P})$ acts regularly on the flags of $\mathcal{P}$, as the flag stabilizers in the automorphism group of any polytope are trivial. We call a polytope $\mathcal{P}$ \emph{chiral\/} if $\Gamma(\mathcal{P})$ has two flag orbits such that any two adjacent flags are in distinct orbits (see~\cite{P2025,SW1991}). The group $\Gamma(\mathcal{P})$ of a regular or chiral polytope $\mathcal{P}$ has a well-behaved system of \emph{distinguished generators\/} which can be described as follows. In Section~\ref{grour} we will show more generally that the automorphism group of every polytope with at most two flag orbits has a distinguished generating system. 

If $\mathcal{P}$ is a regular $n$-polytope, then $\Gamma(\mathcal{P})$ is generated by involutions $\rho_0,\ldots,\rho_{n-1}$, where $\rho_i$ maps a fixed, or \emph{base\/}, flag $\Phi$ to its $i$-adjacent flag $\Phi^i$, that is,
\[ \Phi\rho_i = \Phi^i.\]
These generators satisfy (at least) the standard Coxeter-type relations for Coxeter groups with string diagrams, 
\begin{equation}
\label{standardrel}
(\rho_i \rho_j)^{p_{ij}} = 1, \textrm{ for } i,j=0, \ldots,n-1,
\end{equation}
where $p_{ii}=1$, $p_{ji} = p_{ij} =: p_{i+1}$ if $j=i+1$, and $p_{ij}=2$ otherwise. The numbers $p_j$ determine the \emph{Schl\"afli symbol} $\{p_{1},\ldots,p_{n-1}\}$ of $\mathcal{P}$. We also say that $\mathcal{P}$ is of (\emph{Schl\"afli}) \emph{type} $\{p_{1},\ldots,p_{n-1}\}$. Moreover, the following {\em intersection property\/} holds,
\begin{equation}
\label{intprop}
\langle \rho_i \mid i \in K \rangle \cap \langle \rho_i \mid i \in J \rangle
= \langle \rho_i \mid i \in {K \cap J} \rangle,
\textrm{ for } K,J \subseteq \{0,1,\ldots,n-1\}.
\end{equation}

For a regular $n$-polytope $\mathcal{P}$, the elements $\sigma_{1},\ldots,\sigma_{n-1}$ defined by $\sigma_{i}:=\rho_{i-1}\rho_{i}$ for $i=1,\ldots,n-1$ generate the {\em rotation subgroup\/} $\Gamma^{+}(\mathcal{P})$ of $\Gamma(\mathcal{P})$, which is of index at most~$2$ in $\Gamma(\mathcal{P})$. Note that $\Gamma^{+}(\mathcal{P})$ consists of all elements of $\Gamma(\mathcal{P})$ which can be written as a product of an even number of distinguished generators $\rho_i$. We call a regular polytope $\mathcal{P}$ {\em directly regular\/} (or orientably regular) if the index of $\Gamma^{+}(\mathcal{P})$ in $\Gamma(\mathcal{P})$ is $2$.

Let $\Gamma$ be a group with involutory generators $\rho_{0},\ldots,\rho_{n-1}$ that satisfy (at least) the standard Coxeter-type relations (for any Coxeter diagram), 
\begin{equation}
\label{standardrel2}
(\rho_i \rho_j)^{p_{ij}} = 1 \;\; \textrm{ for } i,j=0, \ldots,n-1,
\end{equation}
where $p_{ii}=1$ and $p_{ji} = p_{ij}\geq 2$ for $i\neq j$. Then $\Gamma$ is called a {\it C-group\/} if $\Gamma$ and its generators satisfy the intersection property (\ref{intprop}). Each Coxeter group is a C-group. The ``C'' in C-group stands for ``Coxeter'', though not every C-group is a Coxeter group.  A {\it string C-group\/} is a C-group whose underlying Coxeter diagram is a string; that is, $p_{ii}=1$, $p_{ji} = p_{ij}\geq 2$ for $i\neq j$, and $p_{ij}=2$ if $|i-j|\geq 2$. The automorphism group of every regular polytope is a string C-group. In fact, the string C-groups are precisely the automorphism groups of regular polytopes, since, in a natural way, a regular polytope can be constructed (uniquely) from a string C-group $\Gamma$ and its generators $\rho_{0},\ldots,\rho_{n-1}$ (see \cite[Ch. 2E]{MS2002}). We usually identify a regular polytope with its automorphism (string C-) group.

We write $[p_{1},p_{2},\ldots,p_{n-1}]$ for the Coxeter group whose underlying Coxeter diagram is a string with $n$ nodes and with $n-1$ branches labeled $p_{1},p_{2},\ldots,p_{n-1}$. Here we regard the $i$-th branch of the string as missing if $p_{i}=2$. This group is the automorphism group of the \emph{universal} regular $n$-polytope $\{p_{1},\ldots,p_{n-1}\}$ (see \cite[Ch. 3D]{MS2002}). 

If $\mathcal{P}$ is a chiral $n$-polytope, then its group $\Gamma(\mathcal{P})$ is generated by elements $\sigma_1,\ldots,\sigma_{n-1}$ associated with a base flag $\Phi$ as follows. The generator $\sigma_i$ fixes the faces in $\Phi \setminus \{\Phi_{i-1},\Phi_{i}\}$ and cyclically permutes (``rotates"), or shifts by one step, consecutive $i$-faces of $\mathcal{P}$ in the (polygonal) section $\Phi_{i+1}/\Phi_{i-2}$ of rank $2$, according as $\Phi_{i+1}/\Phi_{i-2}$ is finite or infinite. By replacing a generator by its inverse if need be, we can further achieve that 
\[\Phi\sigma_{i}=\Phi^{i,i-1}.\] 
The resulting generators $\sigma_1,\ldots,\sigma_{n-1}$ of $\Gamma(\mathcal{P})$ then satisfy (at least) the relations
\begin{equation}
\label{chiralrel}
\sigma_i^{p_i} = (\sigma_i\sigma_{i+1}\cdot\ldots\cdot\sigma_j)^{2} = 1,
\textrm{ for } i,j=1,\dots,n-1,  \textrm{ with } i<j,
\end{equation}
where as before the numbers $p_i$ determine the {\it type\/} (and Schl\"afli symbol) $\{p_1,\ldots,p_{n-1}\}$ of~$\mathcal{P}$. Note that the relations in (\ref{chiralrel}) are just the standard relations for the rotation subgroup of the Coxeter group $[p_{1},p_{2},\ldots,p_{n-1}]$ determined by the corresponding Schl\"afli symbol. The intersection property for the groups of chiral polytopes is more complicated than that for string C-groups (see \cite{SW1991}), and we shall describe it later in a more general context. 

For a chiral $n$-polytope $\mathcal{P}$ we set 
\[\Gamma^{+}(\mathcal{P}):=\Gamma(\mathcal{P}).\]
Thus, for a chiral polytope $\mathcal{P}$, the {\it rotation subgroup\/} $\Gamma^{+}(\mathcal{P})$ of $\Gamma(\mathcal{P})$ is $\Gamma(\mathcal{P})$ itself.

The rotation subgroups of directly regular polytopes share many properties with the full automorphism groups of chiral polytopes. The distinguishing factor in the case of directly regular polytopes is the presence of certain involutory group automorphisms for the rotation subgroup. More precisely, if $\mathcal{P}$ is a directly regular or chiral $n$-polytope, then $\mathcal{P}$ directly regular if and only if the rotation subgroup $\Gamma^{+}({\cal P})$ of $\Gamma(\mathcal{P})$ admits an involutory group automorphism mapping the set of generators $\sigma_1,\ldots,\sigma_{n-1}$ to the new set of generators 
\begin{equation}
\label{adjgen}
\sigma_{1}^{-1},\sigma_{1}^{2}\sigma_2,\sigma_3,\ldots,\sigma_{n-1},
\end{equation} 
respectively. Note that in either case, directly regular or chiral, the generators in (\ref{adjgen}) are the distinguished generators of $\Gamma^{+}({\cal P})$ with respect to the $0$-adjacent flag $\Phi^0$ of $\Phi$, instead of $\Phi$, chosen as the base flag of $\mathcal{P}$. 

In the case of a directly regular polytope $\mathcal{P}$, the above involutory group automorphism of $\Gamma^{+}({\cal P})$ is induced by conjugation with the generator $\rho_0$ in the full automorphism group $\Gamma(\mathcal{P})$; note here that, since $\mathcal{P}$ is directly regular, $\rho_0$ does not belong to $\Gamma^{+}({\cal P})$. More generally, if $\mathcal{P}$ is directly regular, conjugation with any generator $\rho_j$ induces a similar group automorphism of $\Gamma^{+}(\mathcal{P})$. On the other hand, for a chiral $n$-polytope $\mathcal{P}$, the two flag orbits yield two sets of generators for $\Gamma^{+}({\cal P})$ which are not conjugate in $\Gamma(\mathcal{P})=\Gamma^{+}({\cal P})$. Thus a chiral polytope occurs in two \emph{enantiomorphic} (mirror image) forms.

An $n$-polytope $\mathcal{P}$ is called {\em $i$-face transitive} if its automorphism group $\Gamma(\mathcal{P})$ acts transitively on the $i$-faces. An $n$-polytope $\mathcal{P}$ is said to be {\em fully-transitive} if $\mathcal{P}$ is $i$-face transitive for each $i =0,\dots,n-1$. Regular and chiral polytopes are examples of fully-transitive polytopes, but there are also others, as we will see in the next section.
\smallskip 

Later we frequently require the following technical lemma concerning transitivity properties of certain subgroups in automorphism groups of polytopes on families of flags.

\begin{lemma}
\label{rhoex}
Let $\mathcal{P}$ be an $n$-polytope and $\Phi$ a flag of $\mathcal{P}$. Let $K\subseteq N:=\{0,\ldots,n-1\}$, and let $\Phi_K$ denote the set of faces in $\Phi$ with ranks in $K$. Suppose that for each $i\in\bar{K}:= N\setminus K$ there exists an automorphism $\rho_i$ of $\mathcal{P}$ with $\Phi\rho_{i}=\Phi^{i}$. Then the subgroup $\langle\rho_{i}\,|\,i\in \bar{K}\rangle$ of $\Gamma(\mathcal{P})$ acts transitively on the set of flags of $\mathcal{P}$ that contain $\Phi_K$.
\end{lemma}

\begin{proof}
We adapt the arguments of the proof of~\cite[Prop.\,2B4]{MS2002}. Let $\Psi$ be any flag of $\mathcal{P}$ containing the subchain $\Phi_K$ of $\Phi$. By the strong flag-connectedness of $\mathcal{P}$, there exist elements $i_1, i_2, \dots, i_l \in \bar{K}$ such that $\Psi=\Phi^{i_1, i_2, \dots, i_l}$. We now proceed by induction on $l$ to show that 
\begin{equation}
\label{rhoaction}
\Phi^{i_1,\dots,i_{l-1},i_l} = \Phi\rho_{i_l}\rho_{i_{l-1}}\dots \rho_{i_1},
\end{equation}
which in turn implies that $\Psi$ and $\Phi$ must lie in the same orbit under $\langle\rho_{i}\,|\,i\in \bar{K}\rangle$. For $l=1$ the statement in (\ref{rhoaction}) is guaranteed to hold by our assumptions on $\Phi$; in fact, in this case $\Phi^{i_1}=\Phi\rho_{i_1}$ and we are done. Now suppose inductively that (\ref{rhoaction}) holds for an integer $l\geq 1$. Then, by (\ref{auto-adjacent}), 
\[\begin{array}{lllllllll}
\Phi^{i_1,\dots, i_l,i_{l+1}}\!&\!\!=\!\!&\!(\Phi^{i_1,\dots, i_l})^{i_{l+1}} \!&\!\!=\!\!&\! (\Phi\rho_{i_{l}}\dots \rho_{i_1})^{i_{l+1}}\!&&& \\[.03in]
&&&\!\!=\!\! &\!(\Phi^{i_{l+1}})\rho_{i_{l}}\dots \rho_{i_1}\! &\!\!=\!\!&\! (\Phi\rho_{i_{l+1}})\rho_{i_{l}}\dots \rho_{i_1}\!&\!\!=\!\!&\! \Phi\rho_{i_{l+1}}\rho_{i_{l}}\dots \rho_{i_1},
\end{array}\]
and again we are done. Thus the equation (\ref{rhoaction}) can be established by repeated application of~(\ref{auto-adjacent}). This settles the lemma.
\end{proof}

The reader should observe that the superscripts for the flags on the left side of (\ref{rhoaction}) occur in reverse order as the subscripts of the generators in the product on the right side of~(\ref{rhoaction}). Thus left multiplication of an automorphism $\alpha\in\Gamma(\mathcal{P})$ by a generator $\rho_i$ has the following effect:\ if $\Phi\alpha=\Phi^t$ for some sequence of superscripts $t$, then by a slight abuse of notation, 
\[ \Phi\rho_{i}\alpha = \Phi^{t,i} .\]

Throughout the paper, we frequently make use of the fact that an automorphism of a polytope is uniquely determined by its effect on a single flag.

\section{Two-orbit polytopes}
\label{twoorbitspols}

An $n$-polytope $\mathcal{P}$ is said to be a {\em two-orbit polytope} if its automorphism group $\Gamma(\mathcal{P})$ has exactly two orbits on the flags of $\mathcal{P}$. 

Chiral polytopes are particular examples of two-orbit polytopes. Among two-orbit polytopes, chiral polytopes are characterized by the property that any two adjacent flags are in distinct orbits. But unlike for chiral polytopes, there is a priori no condition on a generic two-orbit polytope that requires certain pairs of adjacent flags to be in the same or in different orbits under $\Gamma(\mathcal{P})$.

All polytopes of rank $2$ (polygons) are regular, so two-orbit polytopes must necessarily have rank at least $3$. The well-known cuboctahedron and icosidodecahedron and their duals, the rhombic dodecahedron and rhombic triacontahedron respectively, as well as their Petrials, are simple examples of non-chiral two-orbit polyhedra. For a general investigation of two-orbit polyhedra we refer to \cite{H2010}.
\bigskip

\begin{figure}[htbp]
\begin{center}
\includegraphics[width=3.4cm]{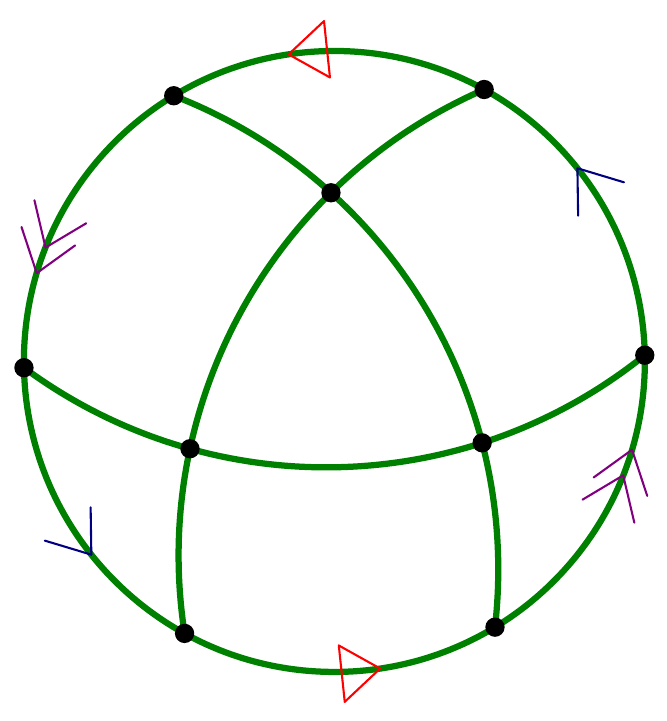} \qquad\quad
\includegraphics[width=3.4cm]{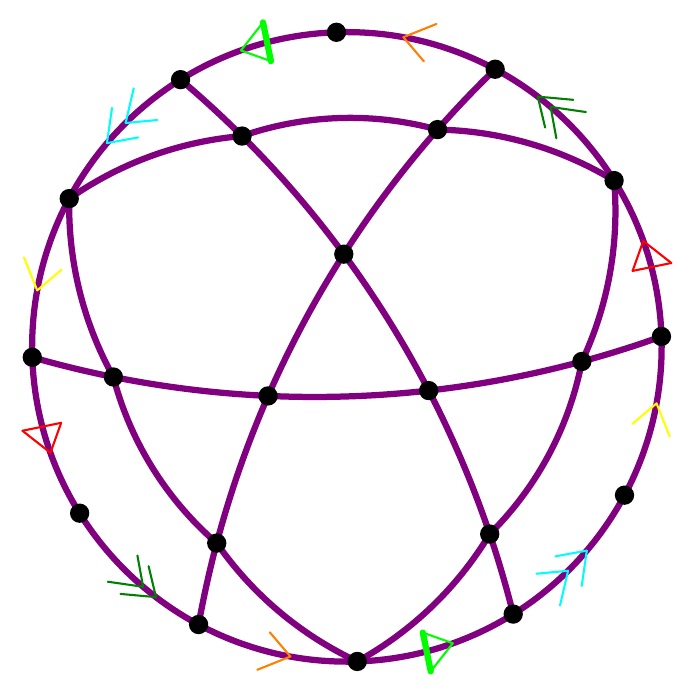}
\caption{The hemi-cuboctahedron and hemi-icosidodecahedron are two-orbit polyhedra in the projective plane which are not chiral.}
\label{hemi}
\end{center}
\end{figure}

We begin our study of two-orbit polytopes with the following key observation which exploits the fact that there are precisely two flag orbits (see \cite[Lemma 2]{H2010}). The short proof is included for completeness.

\begin{lemma}
\label{twoorbits}
Let $\mathcal{P}$ be a two-orbit polytope of rank $n$, let $\Phi,\Psi$ be any flags of $\mathcal{P}$, and let $0\leq i\leq n-1$. If $\Phi$ and $\Phi^i$ lie in the same flag orbit under $\Gamma(\mathcal{P})$, then $\Psi$ and $\Psi^i$ also lie in the same flag orbit under $\Gamma(\mathcal{P})$.
\end{lemma}

\begin{proof}
Suppose that $\Phi$ and $\Phi^i$ lie in the same flag orbit under $\Gamma(\mathcal{P})$ but that $\Psi$ and $\Psi^i$ lie in a different flag orbits under $\Gamma(\mathcal{P})$. Since the polytope $\mathcal{P}$ has just two flag orbits, one of $\Psi$ or $\Psi^i$ must lie in the same flag orbit as both $\Phi$ and $\Phi^i$. If $\Psi$ itself lies in the same flag orbit as $\Phi$ and thus $\Psi=\Phi\alpha$ for some $\alpha\in\Gamma(\mathcal{P})$, then (\ref{auto-adjacent}) shows that  
\[ \Psi^{i} = (\Phi\alpha)^{i}= (\Phi^{i})\alpha.\]
Thus, since the orbits of $\Phi$ and $\Phi^i$ are the same, both $\Psi$ and $\Psi^i$ lie in the same orbit as $\Phi$ and $\Phi^i$. This is a contradiction. The arguments for the case when $\Psi^i$ lies in the same flag orbit as $\Phi$ are similar.
\end{proof}

As a consequence of Lemma~\ref{twoorbits}, two-orbit $n$-polytopes naturally fall into different {\em classes\/}, each indexed by a {\em proper\/} subset $I$ of 
\[ N:=\{0, \dots, n-1\}  \]
called the {\it class type set\/} of the class.
(Throughout we use $\subseteq$ to indicate inclusion of sets, and $\subset$ to indicate strict inclusion of sets.) For $I\subset N$, a two-orbit $n$-polytope $\mathcal{P}$ is said to belong to the {\em class $2_I$} if $I$ contains precisely those elements $i$ of $N$ for which a flag and its $i$-adjacent flag are in the same flag orbit under $\Gamma(\mathcal{P})$.

By Lemma~\ref{twoorbits}, the class $2_I$ of a two-orbit polytope is well-defined. Note that the chiral $n$-polytopes are precisely the two-orbit $n$-polytopes in the class $2_{\emptyset}$, obtained for the extreme case $I = \emptyset$.

We also define the class $2_N$ to consist of all regular, or {\em one-orbit\/}, $n$-polytopes. This is reasonable as a polytope $\mathcal{P}$ for which each flag lies in the same flag orbit as all of its adjacent flags must necessarily have one flag orbit under $\Gamma(\mathcal{P}$ and hence be regular (see \cite[Prop. 2B4]{MS2002}). This also explains why, for a two-orbit polytope, its class type set $I$ must be a proper subset of $N$ and thus have at most $n-1$ elements. 

It is convenient to introduce notation for the complement of a subset $J$ of $N$ in $N$. For $J\subseteq N$ define 
\[\bar{J}:=N \setminus I.\]
The cardinality of $\bar{J}$ is called the {\it rank deficiency} of $J$ and is denoted ${\rm rd}(J)$. Thus
\[ {\rm rd}(J) := n- |J|.\]
We apply this concept primarily for class type sets $I$ of two-orbit polytopes.

For a two-orbit $n$-polytope or regular $n$-polytope $\mathcal{P}$, the rank deficiency of its class type set $I$ is called the {\it reflection deficiency} of $\mathcal{P}$ and is denoted ${\rm rd}(\mathcal{P})$. Thus,
\[ {\rm rd}(\mathcal{P}) =|\bar{I}|.\]
Informally speaking, for a two-orbit $n$-polytope $\mathcal{P}$ in the class $2_I$, the reflection deficiency of $\mathcal{P}$ is the total number of generators $\rho_i$ missing from a standard generating set for the automorphism group $\Gamma(\mathcal{P})$ of a regular $n$-polytope. Regular polytopes have reflection deficiency $0$, and chiral polytopes have reflection deficiency $n$. A typical two-orbit polytope has reflection deficiency at least 1 and lies in between these two extreme cases. The structure of a two-orbit polytope depends significantly on its reflection deficiency. We will see that two-orbit polytopes with reflection deficiency values of $1$, $2$, and $\geq 3$ behave quite differently. 

Note that the class type set $I$ completely encodes the data about the local configurations of flags belonging to the same orbit. For illustrations in rank 3, see \cite[Figures 1,\,2]{H2010}. In rank 3 there are seven classes of two-orbit polyhedra. The polyhedra in the five classes $2_\emptyset$, $2_{\{0\}}$, $2_{\{1\}}$, $2_{\{2\}}$ or $2_{\{0,2\}}$ are {\em equivelar\/} of some type $\{p,q\}$, meaning that their facets are $p$-gons and their vertices are $q$-valent for some $p$ and $q$. In particular, $p$ must be even for the polyhedra in $2_{\{0\}}$, $2_{\{1\}}$ or $2_{\{0,2\}}$, and $q$ must be even for those in  $2_{\{1\}}$, $2_{\{2\}}$ or $2_{\{0,2\}}$. Equivelarity generally breaks down for the two remaining classes, $2_{\{0,1\}}$ and $2_{\{1,2\}}$, as the following examples show. The polyhedra in the class $2_{\{0,1\}}$ are vertex transitive; the cuboctahedron and icosidodecahedron lie in this class and have two kinds of faces, triangles and squares or pentagons, respectively. Dually, the polyhedra in $2_{\{1,2\}}$ are face transitive; examples are the rhombic dodecahedron and rhombic triacontahedron which have vertices of degree 3 and 4 or 5, respectively.
\smallskip

Before proceeding, we note an important consequence of the definition of the class type set $I$ for sequences of successively adjacent flags of two-orbit polytopes such as 
\[\Phi,\Phi^{i_1},\Phi^{i_{1},i_{2}}, \ldots, \Phi^{i_{1},\ldots,i_{l}}.\]
So, let $\mathcal{P}$ be a two-orbit $n$-polytope in the class $2_I$.  As we move along the flag sequence of $\mathcal{P}$ from $\Phi$ to $\Phi^{i_{1},\ldots,i_{l}}$, we change flag orbits under $\Gamma(\mathcal{P})$ each time we encounter a superscript that does not belong to $I$, but otherwise leave flag orbits unchanged. In particular this permits us to determine whether or not the first flag $\Phi$ and last flag $\Phi^{i_{1},\ldots,i_{l}}$ belong to the same flag orbit under $\Gamma(\mathcal{P})$. For example, for all $i \in I$ and $j,k \notin I$, the flags $\Phi^i$, $\Phi^{j,k}$ and $\Phi^{j,i,j}$ all are in the same flag orbit under $\Gamma(\mathcal{P})$ as $\Phi$, as each can be joined to $\Phi$ by a sequence of successively adjacent flags in which overall an even number of superscripts not contained in $I$ occurs. 

As a direct consequence of these considerations (or alternatively, of Lemma~\ref{rhoex}) we mention the following useful lemma which deals with a particularly interesting special case. 

\begin{lemma}
\label{omegaorb}
Let $\mathcal{P}$ be a two-orbit $n$-polytope in the class $2_I$, with $I\subset N$, and let $\Omega$ be a chain of $\mathcal{P}$. Let $t(\Omega)$ denote the set of ranks of the proper faces of $\Omega$, and suppose that $\bar{I}\subseteq t(\Omega)$. Then any two flags of $\mathcal{P}$ containing $\Omega$ lie in the same flag orbit under $\Gamma(\mathcal{P})$.
\end{lemma}

\begin{proof}
If $\Phi$ and $\Psi$ are two flags containing $\Omega$, then the strong flag-connectedness of $\mathcal{P}$ gives a sequence of successively adjacent flags, all containing $\Omega$, which joins $\Phi$ and $\Psi$ in such a way that all successive flag-adjacencies in this sequence occur at ranks not contained in $t(\Omega)$. Hence, as $\bar{I}\subseteq t(\Omega)$, all successive flag-adjacencies must occur at ranks contained in $I$. Now the previous considerations show that successive flags in the sequence, and thus $\Phi$ and $\Psi$, lie in the same orbit under $\Gamma(\mathcal{P})$.
\end{proof}

Two-orbit polytopes, by definition, have just two flag-orbits. This immediately raises the question about the number of face-orbits for the faces of any given rank in a two-orbit polytope. Clearly, a two-orbit polytope can have at most two orbits on the faces of any rank, since any automorphism that maps a flag to another flag, also maps the $i$-face of the first flag to the $i$-face of the second flag for any rank $i$. 

The following theorem, taken from \cite{H2010}, answers the question and describes completely the face transitivity properties of the automorphism group of a two-orbit polytope. Recall that a polytope is said to be fully-transitive if its automorphism group is transitive on the faces of each rank.

\begin{theorem}
\label{almost-fully-transitive}
Let $\mathcal{P}$ be a two-orbit $n$-polytope in the class $2_I$, with $I \subset N$. \\[.02in]
(a)\ Then $\mathcal{P}$ is fully-transitive if and only if ${\rm rd}(I)>1$. \\[.02in]
(b)\ If ${\rm rd}(I)=1$ and $\bar{I} = \{j\}$ for some $j\in N$, then $\mathcal{P}$ is $i$-face transitive for every $i$ in $N$ with $i \neq j$; moreover, $\mathcal{P}$ has exactly two orbits on the set of $j$-faces and these are represented by the $j$-faces in any pair of $j$-adjacent flags.
\end{theorem}

\begin{proof}
For the reader's convenience we are including the proof, following the arguments in \cite[pp.\,947-948]{H2010}. Let $\Phi:=\{\Phi_0,\ldots,\Phi_{n-1}\}$ be the base flag of $\mathcal{P}$. 

We first show that if $\mathcal{P}$ is not $j$-face transitive for some $j\in N$, then necessarily ${\rm rd}(I)=1$ (that is, $|I|=n-1$) and $I=N\setminus\{j\}$. Suppose $\mathcal{P}$ is not $j$-face transitive. Define $F:=\Phi_j$, and let $G$ be a $j$-face of $\mathcal{P}$ such that $F$ and $G$ lie in different $j$-face orbits under $\Gamma(\mathcal{P})$. Then any two flags $\Lambda,\Psi$ with  $F\in\Lambda$ and $G\in\Psi$ must necessarily lie in different flag orbits, as otherwise their $j$-faces $F,G$ would lie in the same $j$-face orbit. As $\mathcal{P}$ has only two flag orbits, it follows that the flags $\Lambda$ with $F\in\Lambda$ must all lie in the same flag orbit, namely the flag orbit as $\Phi$. In particular, for each $i\in N$ with $i\neq j$, the $i$-adjacent flag $\Phi^i$, which contains $F$, must lie the same orbit as~$\Phi$. Thus, by the definition of $I$, we must have $I= N\setminus\{j\}$. 

It follows that $\mathcal{P}$ must be fully-transitive whenever ${\rm rd}(I)>1$. It remains to investigate the case when ${\rm rd}(I)=1$ and $I=N\setminus\{j\}$ for some $j\in N$. 

In this case we begin by proving that $\mathcal{P}$ is $i$-face transitive for each $i\in I$. Our arguments actually work for any choice of $I$, but when $I=N\setminus\{j\}$ they immediately imply that $\mathcal{P}$ is $i$-face transitive for all ranks $i\in N$ with $i\neq j$. 

So let $i\in I$, and let $F$ be an $i$-face of $\mathcal{P}$. Let $\Phi$ be a flag containing $F$, and let $\Phi=\Lambda_0,\Lambda_1,\ldots,\Lambda_m=\Psi$ be a sequence of successively adjacent flags joining $\Phi$ and $\Psi$. Now, considering the sequence of $i$-faces of the flags in the flag sequence, we observe that any two successive $i$-faces either coincide or are the $i$-faces in a pair of $i$-adjacent flags. By deleting duplicates, this gives a sequence in which any two successive $i$-faces are the $i$-faces in a pair of $i$-adjacent flags. Since $i\in I$, any pair of $i$-adjacent flags lies in the same flag orbit under $\Gamma(\mathcal{P})$ and thus their $i$-faces must lie in the same $i$-face orbit under $\Gamma(\mathcal{P})$. It follows that any two successive $i$-faces in the $i$-face sequence lie in the same orbit under $\Gamma(\mathcal{P})$, and therefore that $F$ lies in the same orbit as $\Phi_i$, the $i$-face of $\Phi$. Thus $\mathcal{P}$ is $i$-face transitive.

It remains to prove that $\Gamma(\mathcal{P})$ has exactly two orbits on the $j$-faces of $\mathcal{P}$. We show that otherwise the flags $\Phi$ and $\Phi^j$ would lie in the same orbit under $\Gamma(\mathcal{P})$, which is impossible since $j\notin I$. 

So, suppose to the contrary that $\mathcal{P}$ is $j$-face transitive. Then the $j$-faces of the flags $\Phi$ and $\Phi^j$ are in the same orbit, so there exists $\gamma\in\Gamma(\mathcal{P})$ such that $\Phi_j\gamma=\Phi^{j}_{j}$. Let $\Psi:=\Phi\gamma$ and consider the sequence of successively adjacent flags $\Phi^j=\Lambda_0,\Lambda_1,\ldots,\Lambda_m=\Psi$, all containing $\Phi^{j}\cap\Psi$, joining the flags $\Phi^j$ and $\Psi$. Note that $\Phi^{j}_{j}=\Psi_j$, and that this $j$-face lies in $\Phi^{j}\cap\Psi$ and thus is shared by every flag of the sequence. It follows that any two successive flags of the sequence are adjacent at ranks different from $j$ (that is, at ranks contained in $I$) and therefore must lie in the same flag orbit under $\Gamma(\mathcal{P})$. This shows that $\Phi^j$ and $\Psi$ belong to the same flag orbit. Thus $\Phi^j$ and $\Phi$ are in the same flag orbit, which is a contradiction.
\end{proof}
\medskip

Our next theorem concerns the action of the automorphism group $\Gamma(\mathcal{P})$ on the sections of a two-orbit polytope $\mathcal{P}$, and completely describes the transitivity properties of $\Gamma(\mathcal{P})$ on sets of comparable sections.

For an $n$-polytope $\mathcal{P}$ and ranks $i$ and $j$ with $-1 \leq i \leq j \leq n$, let $\mathcal{S}_{i,j}=\mathcal{S}_{i,j}(\mathcal{P})$ denote the set of all sections $G/F$ of $\mathcal{P}$ that are determined by an incident pair of an $i$-face $F$ and a $j$-face $G$ of $\mathcal{P}$. Clearly, $\Gamma(\mathcal{P})$ maps sections in $\mathcal{S}_{i,j}$ to sections in $\mathcal{S}_{i,j}$; and if $\Phi$ and $\Psi$ are flags in the same orbit under $\Gamma(\mathcal{P})$, then the corresponding sections $\Phi_{j}/\Phi_{i}$ and $\Psi_{j}/\Psi_{i}$ in $\mathcal{S}_{i,j}$ are also in the same orbit under $\Gamma(\mathcal{P})$.
\smallskip

\begin{theorem}
\label{sections}
Let $\mathcal{P}$ be a two-orbit $n$-polytope in the class $2_I$, with $I \subset N$, and let $-1 \leq i \leq j \leq n$.\\[.02in] 
(a)\ If $\bar{I} \not \subseteq \{i,j\}$, then $\Gamma(\mathcal{P})$ acts transitively on 
$\mathcal{S}_{i,j}(\mathcal{P})$; in particular, any two sections in $\mathcal{S}_{i,j}(\mathcal{P})$ are isomorphic.\\[.02in]
(b)\ If $\bar{I} \subseteq \{i,j\}$, then $\Gamma(\mathcal{P})$ has two orbits on $\mathcal{S}_{i,j}(\mathcal{P})$ and these are represented by the two sections $\Phi_{j}/\Phi_{i}$ and $\Phi_{j}^{k}/\Phi_{i}^{k}$ of $\mathcal{P}$, where $k$ is any element of $N$ with $k\notin I$ and $\Phi,\Phi^k$ is any pair of $k$-adjacent flags of $\mathcal{P}$.
\end{theorem}

\begin{proof}
First, let $\bar{I} \not \subseteq \{i,j\}$. Choose an element $k\notin I$ such that $k\neq i,j$. Then, by the definition of the class $2_I$, a flag and its $k$-adjacent flag are in distinct orbits. If $\Phi$ is any flag of $\mathcal{P}$, then $\Phi^k$ shares the same $i$-face and $j$-face with $\Phi$, so in particular, 
\[\Phi_j / \Phi_i =\Phi^k_j / \Phi^k_i.\]
Now if $\Psi$ is any flag of $\mathcal{P}$, then $\Psi$ must be equivalent to $\Phi$ or $\Phi^k$ under $\Gamma(\mathcal{P})$. 
Hence if $\gamma \in \Gamma(\mathcal{P})$ is such that $\Psi=\Phi\gamma$ or $\Psi=\Phi^k\gamma$, then either way, $\Psi_j/\Psi_{i}=(\Phi_j/ \Phi_i)\gamma$.  This proves the first part of the lemma, as each section in $\mathcal{S}_{i,j}(\mathcal{P})$ is of the form $\Psi_j/\Psi_i$ for some flag $\Psi$ of $\mathcal{P}$. 

Now let $\bar{I} \subseteq \{i,j\}$. Choose again an element $k\notin I$; then necessarily $k=i$ or $k=~j$. If $\Phi$ is any flag of $\mathcal{P}$, then every section in $\mathcal{S}_{i,j}(\mathcal{P})$ is equivalent under $\Gamma(\mathcal{P})$ to either $\Phi_j / \Phi_i$ or $\Phi^k_j / \Phi^k_i$. However, now the latter two sections are not equivalent under $\Gamma(\mathcal{P})$, as can be seen as follows. Suppose to the contrary that there exists an element $\gamma\in\Gamma(\mathcal{P})$  such that $(\Phi_j / \Phi_i)\gamma=\Phi^k_j / \Phi^k_i$. Then the flags $\Psi:=\Phi\gamma$ and $\Phi$ are in the same orbit and 
\[\Psi_{i}=(\Phi\gamma)_{i} = (\Phi_i)\gamma=\Phi^k_i,\quad 
\Psi_{j}=(\Phi\gamma)_{j}=(\Phi_j)\gamma=\Phi^k_j.\] 
Thus the flags $\Psi$ and $\Phi^k$ have the same $i$-faces and the same $j$-faces. Hence, by the strong flag-connectedness, $\Psi$ and $\Phi^k$ can be joined by a sequence of successively adjacent flags, where at each step in the sequence the adjacency occurs at a rank different from $i$ and $j$. Since $\bar{I}\subseteq\{i,j\}$ and thus each pair of successively adjacent flags in the sequence lies in the same orbit under $\Gamma(\mathcal{P})$, it follows that $\Psi$ and $\Phi^k$ also belong to the same orbit. Thus $\Phi$ and $\Phi^k$ lie in the same flag orbit under $\Gamma(\mathcal{P})$. Hence $k\in I$, by the definition of $I$; this is a contradiction to our choice of $k$. It follows that there are exactly two orbits on $\mathcal{S}_{i,j}(\mathcal{P})$ when $\bar{I}\subseteq \{i,j\}$. This proves the second part of the lemma.
\end{proof}
\medskip

Schl\"afli symbols are a convenient way of encoding the local combinatorial structure in  regular, chiral, or more generally, equivelar polytopes. Recall that an $n$-polytope $\mathcal{P}$ is said to be \textit{equivelar of type $\{p_1,\ldots,p_{n-1}\}$}, or simply, to be \textit{of type $\{p_1,\ldots,p_{n-1}\}$}, if for each $i=1,\ldots,n-1$, every 2-section in $\mathcal{S}_{i-2,i+1}(\mathcal{P})$ is isomorphic to a $p_i$-gon if $p_i<\infty$ or to an apeirogon if $p_{i}=\infty$. As the example of the cuboctahedron shows, two-orbit polytopes may not be equivelar and may not have a standard (one-row) Schl\"afli symbol. In general, two-orbit polytopes require a two-row Schl\"afli symbol to encode the local combinatorial structure.

More precisely, if $\mathcal{P}$ is a two-orbit $n$-polytope, we can associate with $\mathcal{P}$ a {\em double\/} (or {\em two-row\/}) {\em Schl\"afli symbol\/}, 
\begin{equation}
\label{doubleSch1}
\Big\{ 
\begin{matrix} 
p_1 &\! p_2 &\!\ldots &\! p_{n-1} \\
q_1 &\! q_2 &\! \ldots &\! q_{n-1}
\end{matrix} 
\Big\},
\end{equation}
which is unique up to interchanging the rows. Let $\mathcal{O}$ and $\mathcal{O}'$ be the two flag orbits of $\mathcal{P}$. If $\Phi$ is  flag, the number of $i$-faces of $\mathcal{P}$ in a 2-section $\Phi_{i+1}/ \Phi_{i-2}$ clearly only depends on the orbit of a flag $\Phi$.  We denote this number by $p_{i}(\mathcal{O})$ and $q_{i}(\mathcal{O})$ according as $\Phi$ lies in $\mathcal{O}$ or $\mathcal{O}'$. Then the double Schl\"afli symbol is a $2\times (n-1)$ array recording the numbers $p_i:=p_i(\mathcal{O})$ in the first row and the numbers $q_i:=q_i(\mathcal{O})$ in the second row, as indicated above. Note that $p_i(\mathcal{O}')= q_i(\mathcal{O})$ and $q_i(\mathcal{O}')=p_i(\mathcal{O})$. Thus interchanging the orbits $\mathcal{O}$ and $\mathcal{O}'$ results in interchanging the rows of the symbol. Symbols obtained from each other by switching the two rows represent two equivalent ways of describing the local structure of $\mathcal{P}$ around flags.

The two rows of the double Schl\"afli symbol frequently coincide. This occurs if and only if $\mathcal{P}$ is equivelar. In this case we usually reduce the symbol to the standard one-row Schl\"afli symbol $\{p_1, \dots p_{n-1}\}$.  For example, equivelarity occurs whenever  either ${\rm rd}(I)\geq 3$ or ${\rm rd}(I)=2$ and $\bar{I} = \{ j, k\}$ for some $j,k$ with $|j-k| \neq 3$. In fact, under these assumptions on~$I$, any two comparable 2-sections are isomorphic by Theorem~\ref{sections}, and therefore $q_i=p_i$ for each $i$. Thus the double Schl\"afli symbol can only have distinct rows when either ${\rm rd}(I)=3$ or ${\rm rd}(I)=2$ and $\bar{I} = \{j, k\}$ for some $j,k$ with $|j-k|=3$.

On the other hand, again by Theorem~\ref{sections}, if $\bar{I}=\{j,k\}$ with $k=j+3$, then still $q_i=p_i$ for each $i$ with $i \neq j+2=k-1$, but now $q_{j+2}$ and $p_{j+2}$ need not be the same. In this case we often replace the full two-row symbol in (\ref{doubleSch1}) by the simpler symbol
\begin{equation}
\label{doubleSch2}
\Big\{p_1,\dots,p_{j+1}, 
\begin{matrix} p_{j+2} \\ q_{j+2} \end{matrix}, 
p_{j+3}, \dots p_{n-1}\Big\}.
\end{equation}
The new symbol is unique up to interchanging $p_{j+2}$ and $q_{j+2}$. In practice we often place the smaller of the two integers $p_{j+2}$ and $q_{j+2}$ in the top row.

Similarly, once again by Theorem~\ref{sections}, if ${\rm rd}(I)=1$ and $\bar{I}=\{j\}$, then $q_i=p_i$ for each $i \neq j-1, j+2$, but when $i=j-1$ or $j+2$ the numbers $q_i$ and $p_i$ may be different. In this case we often use the simpler symbol
\begin{equation}
\label{doubleSch3}
\Big\{p_1, \dots, p_{j-2}, \begin{matrix} p_{j-1} \\ q_{j-1} \end{matrix}, p_{j}, p_{j+1},\begin{matrix} p_{j+2} \\ q_{j+2} \end{matrix}, p_{j+3}, \dots p_{n-1}\Big\}
\end{equation}
in place of the full two-row symbol. This symbol is unique up to interchanging, simultaneously, $p_{j-1}$ with $q_{j-1}$, and $p_{j+2}$ with $q_{j+2}$. Again the symbol with the smallest entries in the top row is usually preferred.

In either of the two scenarios just described we still refer to the new symbol as a double Schl\"afli symbol for $\mathcal{P}$.

For example, the double Schl\"afli symbols of the cuboctahedron and icosidodecahedron, which are two-orbit polyhedra in the class $2_{\{0,1\}}$, are given by 
\[\Big\{\!\!\!\!
\begin{array}{rl}
\begin{array}{l} 3\\ 4\end{array}&\!\!\!\!4
\end{array}
\!\!\Big\},
\;\;
\Big\{\!\!\!\!
\begin{array}{rl}
\begin{array}{l} 3\\ 5\end{array}&\!\!\!\!4
\end{array}
\!\!\Big\},\] 
respectively. Their duals, the rhombic dodecahedron and rhombic triacontahedron, belong to the class $2_{\{0,1\}}$ and have double Schl\"afli symbols
\[\Big\{\!\!\!\!
\begin{array}{rl}
\;4&\!\!\!\!\begin{array}{l} 3\\ 4\end{array}\!\!
\end{array}
\!\!\Big\},
\;\;\,
\Big\{\!\!\!\!
\begin{array}{rl}
\;4&\!\!\!\!\begin{array}{l} 3\\ 5\end{array}\!\!
\end{array}
\!\!\Big\}.\] 

\section{The group of a two-orbit polytope}
\label{grour}

In this section we establish structure results for the automorphism groups of two-orbit polytopes following the blueprint for similar approaches for regular and chiral polytopes.

To this end, throughout this section, $\mathcal{P}$ shall be a two-orbit $n$-polytope in the class $2_I$, $I\subset N=\{0,\ldots,n-1\}$, with double Schl\"afli symbol 
$$\Big\{ 
\begin{matrix} 
p_1 &\! p_2 &\! \dots &\! p_{n-1} \\
q_1 &\! q_2 &\! \dots &\! q_{n-1}
\end{matrix} \Big\}.$$
The Schl\"afli symbol of $\mathcal{P}$ is uniquely determined up to interchanging the top and bottom rows. In our subsequent discussion we usually choose the symbol whose top row aligns with the orbit of a specified flag of $\mathcal{P}$. 

\subsection{Generators}
\label{gensection}

We begin by investigating generators for the automorphism group. Let $\mathcal{P}$ be an $n$-polytope in the class $2_I$, with $I\subset N$, and let $\Phi$ be a fixed, or {\em base\/}, flag of $\mathcal{P}$. We assume that the top row of the Schl\"afli symbol corresponds to the flag orbit that contains $\Phi$.

The automorphism group $\Gamma(\mathcal{P})$ of $\mathcal{P}$ has a natural system of generators obtained as follows. As we saw in the previous section, for all $i \in I$ and $j,k \notin I$, the flags $\Phi^i$, $\Phi^{j,k}$ and $\Phi^{j,i,j}$ all lie in the same flag orbit as $\Phi$, so there exist (unique) elements $\rho_i$, $\alpha_{j,k}$ and $\alpha_{j,i,j}$ in $\Gamma(\mathcal{P})$ 
such that
\begin{equation}
\label{phiaction}
\Phi\rho_{i}=\Phi^{i},\;\; \Phi\alpha_{j,k}=\Phi^{j,k},\;\; \Phi\alpha_{j,i,j}=\Phi^{j,i,j} \quad (i\in I,\ \! j,k\notin I).
\end{equation}
Note that $\alpha_{j,k}=1$ if $j=k$. Relative to the $j$-adjacent flag $\Phi^j$ of $\Phi$, which lies in the flag orbit not containing $\Phi$, a typical element $\alpha_{j,i,j}$ acts like the element $\rho_i$ relative to $\Phi$. More exactly, if $i\in I$ and $j\notin I$, then 
\[ (\Phi^{j})\alpha_{j,i,j} = (\Phi\alpha_{j,i,j})^{j} = (\Phi^{j,i,j})^{j}
= \Phi^{j,i,j,j} = \Phi^{j,i} = (\Phi^{j})^{i}\]
and hence $\alpha_{j,i,j}$ maps the flag $\Phi^{j}$ to its $i$-adjacent flag $(\Phi^{j})^{i}$. In fact, if $j\notin I$, the relationship between the two sets $\{\rho_i\,|\,i\in I\}$ and $\{\alpha_{j,i,j}\,|\,i\in I\}$ is fully symmetric, in that $\{\rho_i\,|\, i\in I\}$ is to the base flag $\Phi$ what $\{\alpha_{j,i,j}\,|\, i\in I\}$ is to its $j$-adjacent flag $\Phi^j$; and also, in that $\{\rho_i\,|\, i\in I\}$ is to $\Phi^j$ what $\{\alpha_{j,i,j}\,|\, i\in I\}$ is to $\Phi$. 

Our first theorem says that  
\begin{eqnarray}
\label{gi}
\mathcal{G}_{I} :=\mathcal{G}_{I}(\Phi):= 
\{ \rho_i \,|\, i \in I\} \,\cup\, \{\alpha_{j,k}\,|\, j,k \notin I   \} \,\cup\, \{ \alpha_{j,i,j} \,|\, i \in I, j \notin I \}
\end{eqnarray}
is a set of generators of $\Gamma(\mathcal{P})$. We usually suppress the reference to $\Phi$ in the notation if no confusion is possible.
\smallskip

Before proceeding with the theorem itself, it is instructive to observe the effect, on flags, of multiplying an automorphism $\alpha\in\Gamma(\mathcal{P})$ on the left by an element $\gamma\in\mathcal{G}_{I}$ of the form $\rho_i$, $\alpha_{j,k}$, or $\alpha_{j,i,j}$. We already made a similar observation for $\rho_i$ at the end of Section~\ref{bano}.

Recall from (\ref{phiaction}) that the three kinds of elements of $\mathcal{G}_{I}$ are defined by specific actions on the base flag $\Phi$. Abusing notation for a moment, for $\gamma\in\mathcal{G}_{I}$, let us write $\Phi\gamma=\Phi^{s(\gamma)}$, where $s(\gamma)$ is the $1$-, $2$-, or $3$-element sequence of superscripts $i$, $jk$, or $jij$, according as $\gamma$ is $\rho_i$, $\alpha_{j,k}$, or $\alpha_{j,i,j}$. Now suppose $\alpha\in\Gamma(\mathcal{P})$ and $\Phi\alpha=\Phi^t$ for some sequence of superscripts~$t$. Then we claim that 
\begin{equation}
\label{append}
\Phi (\gamma\alpha) = (\Phi\alpha)^{s(\gamma)} = \Phi^{t,s(\gamma)}.
\end{equation}
This follows immediately from (\ref{phiaction}) and the fact that automorphisms of polytopes preserve $l$-adjacency of flags for each $l\in N$. In fact, by slight abuse of notation, 
\[ \Phi (\gamma\alpha) = (\Phi\gamma)\alpha = \Phi^{s(\gamma)}\alpha
=(\Phi\alpha)^{s(\gamma)} = (\Phi^{t})^{s(\gamma)} = \Phi^{t,s(\gamma)}.\]
Thus multiplying $\alpha$ {\it on the left} by $\gamma$ translates into appending the sequence of superscripts $s(\gamma)$ for $\gamma$ {\it on the right\/}, to the sequence of superscripts $t$ for $\alpha$.
\smallskip

\begin{theorem}
\label{generators}
Let $\mathcal{P}$ be a two-orbit polytope in the class $2_I$, with $I\subset N$, and let $\mathcal{G}_{I}$ be as in (\ref{gi}). Then $\Gamma(\mathcal{P})$ is generated by $\mathcal{G}_{I}$.
\end{theorem}

\begin{proof}
As above, let $\Phi$ be the base flag of $\mathcal{P}$. Our goal is to write an arbitrary element $\psi$ of $\Gamma(\mathcal{P})$ in terms of the elements of $\mathcal{G}_{I}$. 

Consider the flag $\Psi:=\Phi\psi$, which belongs to the same flag orbit as $\Phi$. By the strong flag connectedness of~$\mathcal{P}$, there exists a sequence of successively adjacent flags 
\begin{eqnarray}
\label{adjacentsequence1}
\Phi, \Phi^{i_1}, \Phi^{i_1, i_2}, \dots, \Phi^{i_1, i_2, \dots, i_l}=\Psi,
\end{eqnarray}
all containing $\Phi \cap \Psi$, joining $\Phi$ and $\Psi$.  For the purpose of this proof define $\Phi^{(0)}:=\Phi$ and $\Phi^{(m)}:= \Phi^{i_1,\dots, i_m}$ for $m=1,\ldots,l$. Our goal is to exploit the structure of the sequence in (\ref{adjacentsequence1}) to produce a factorization of $\psi$ into elements from $\mathcal{G}_{I}$.
 
Since $\Phi$ and $\Psi$ are in the same orbit, the number of superscripts $i_k$ that are not contained in~$I$ must be even, equal to $2s$ (say). Let $i_{k_1}, \ldots, i_{k_{2s}}$ denote the terms in the sequence $i_1, i_2, \ldots, i_l$ that are not contained in~$I$, and set $k_0:=0$ and $k_{2s+1}:=l+1$. Then~(\ref{adjacentsequence1}) takes the form
\begin{eqnarray}
\label{seq2}
\Phi=\Phi^{(k_0)},  \Phi^{(1)}, \dots,  \Phi^{(k_1)},\ldots, \Phi^{(k_2)}, \ldots\ldots, \Phi^{(k_{2s-1})}, \ldots, \Phi^{(k_{2s})}, \ldots, \Phi^{(l)} = \Psi.
\end{eqnarray}
Here the flags $\Phi^{(1)}, \ldots, \Phi^{(k_1-1)}$ at the beginning of the sequence are in the same orbit as~$\Phi$, since all of the original superscripts involved lie in $I$. However, starting with $q=1$, each time we encounter a term of the form $\Phi^{(k_q)}$ as we move along the sequence in (\ref{seq2}), we change from one flag orbit to the other flag orbit. On the other hand, no change of flag orbit occurs at the other terms in (\ref{seq2}).

Now for $q=1, \dots, 2s-1$ define 
\begin{eqnarray*}
\beta_q&\!:=\!&\alpha_{i_{k_q},i_{k_{q+1}}} \alpha_{i_{k_q}, i_{k_{q+1}-1}, i_{k_q}} \alpha_{i_{k_q}, i_{k_{q+1}-2}, i_{k_q}} \ldots\ldots \alpha_{i_{k_q}, i_{k_q+1} i_{k_q}}, \\
\gamma_q&\!:=\!& \rho_{i_{k_{q+1}-1}} \rho_{i_{k_{q+1}-2}} \ldots\ldots \rho_{i_{k_q+1}};
\end{eqnarray*}
here, if $k_q$ and $k_{q+1}$ are consecutive integers then $\beta_{q}=\alpha_{i_{k_q},i_{k_{q+1}}}$ and $\gamma_{q}=1$. We also set
\begin{eqnarray*}
\gamma_{0}   &\!:=\!& \rho_{i_{k_{1}-1}} \rho_{i_{k_{1}-2}} \ldots \rho_{i_1},\\
\gamma_{2s} &\!:=\!& \rho_{i_l} \rho_{i_{l-1}} \dots \rho_{i_{k_{2s}+1}} ,
\end{eqnarray*}
so in particular, $\gamma_{0}=1$ if $k_{1}=1$, and $\gamma_{2s}=1$ if $k_{2s}=l$.
Note that for each $q$, each factor occurring in the expressions for $\beta_q$ or $\gamma_q$ is an element of $\mathcal{G}_{I}$. 

Next we show that the action of the generators of $\mathcal{G}_{I}$ on $\Phi$ gives
\begin{eqnarray*}
\Phi \beta_q      &\!=\!& \Phi^{i_{k_{q}}, i_{k_q+1}, \dots, i_{k_{q+1}-1} , i_{k_{q+1}}} \;\,(q=1,\ldots,2s-1), \\ 
\Phi \gamma_q &\!=\!& \Phi^{i_{k_q+1}, i_{k_q+2}, \dots, i_{k_{q+1}-1}} \;\,(q=0,\ldots,2s).
\end{eqnarray*}
Here the right hand side of the second equation is to be read as $\Phi$ if either $q=0$ and $k_{1}=1$, or $q=2s$ and $k_{2s}=l$. Then the equation for $\Phi\beta_q $ can be established as follows:
\[\begin{array}{lll}
\Phi \beta_q 
&=& \Phi \alpha_{i_{k_q},i_{k_{q+1}}} \alpha_{i_{k_q}, i_{k_{q+1}-1}, i_{k_q}} \alpha_{i_{k_q}, i_{k_{q+1}-2}, i_{k_q}} \ldots\ldots \alpha_{i_{k_q}, i_{k_q+1} i_{k_q}}\\[.05in]
&=& (\Phi^{i_{k_q},i_{k_{q+1}}}) \alpha_{i_{k_q}, i_{k_{q+1}-1}, i_{k_q}} \alpha_{i_{k_q}, i_{k_{q+1}-2}, i_{k_q}} \ldots\ldots \alpha_{i_{k_q}, i_{k_q+1} i_{k_q}}\\[.05in]
&=& (\Phi \alpha_{i_{k_q}, i_{k_{q+1}-1}, i_{k_q}} \alpha_{i_{k_q}, i_{k_{q+1}-2}, i_{k_q}} \ldots\ldots \alpha_{i_{k_q}, i_{k_q+1} i_{k_q}})^{i_{k_q},i_{k_{q+1}}}\\[.05in]
&=& ((\Phi^{i_{k_q}, i_{k_{q+1}-1}, i_{k_q}}) \alpha_{i_{k_q}, i_{k_{q+1}-2}, i_{k_q}} \ldots\ldots \alpha_{i_{k_q}, i_{k_q+1} i_{k_q}})^{i_{k_q},i_{k_{q+1}}}\\[.05in]
&=& ((\Phi \alpha_{i_{k_q}, i_{k_{q+1}-2}, i_{k_q}} \ldots\ldots \alpha_{i_{k_q}, i_{k_q+1} i_{k_q}})^{i_{k_q}, i_{k_{q+1}-1}, i_{k_q}})^{i_{k_q},i_{k_{q+1}}}\\[.05in]
&=& (\Phi \alpha_{i_{k_q}, i_{k_{q+1}-2}, i_{k_q}} \ldots\ldots \alpha_{i_{k_q}, i_{k_q+1} i_{k_q}})^{i_{k_q}, i_{k_{q+1}-1},i_{k_{q+1}}}\\[.05in]
&=& (\Phi \alpha_{i_{k_q}, i_{k_{q+1}-3}, i_{k_q}} \ldots\ldots \alpha_{i_{k_q}, i_{k_q+1} i_{k_q}})^{i_{k_q},i_{k_{q+1}-2},i_{k_{q+1}-1},i_{k_{q+1}}}\\[.02in]
&\vdots&\\[.02in]
&=& \Phi^{i_{k_{q}}, i_{k_q+1}, \dots, i_{k_{q+1}-1} , i_{k_{q+1}}}.
\end{array} \]
The equation for $\Phi \gamma_q$ is more straightforward and follows similarly.

Now define the element $\widehat{\psi}$ of the subgroup $\langle\mathcal{G}_{I}\rangle$ of $\Gamma(\mathcal{P})$ by
\begin{eqnarray}
\label{star}
\widehat{\psi}: =\gamma_{2s} \beta_{2s-1} \gamma_{2s-2} \beta_{2s-3} \ldots\ldots \gamma_4 \beta_3 \gamma_2 \beta_1 \gamma_{0}.
\end{eqnarray}
Then the image of $\Phi$ under $\widehat{\psi}$ is given by
\[ \begin{array}{lll}
\Phi\widehat{\psi} 
&=& (\Phi\gamma_{2s})\,\beta_{2s-1}\gamma_{s-2}\beta_{2s-2}\gamma_{2s-3}\ldots\ldots\beta_{1}\gamma_{0}\\[.05in]
&=& (\Phi^{\,i_{k_{2s}+1},i_{k_{2s}+2},\ldots,i_l})\,\beta_{2s-1}\gamma_{s-2}\beta_{2s-2}\gamma_{2s-3}\ldots\ldots\beta_{1}\gamma_{0}\\[.05in]
&=& (\Phi\beta_{2s-1}\gamma_{s-2}\beta_{2s-2}\gamma_{2s-3}\ldots\ldots\beta_{1}\gamma_{0})^{\,i_{k_{2s}+1},i_{k_{2s}+2},\ldots,i_l}\\[.05in]
&=& ((\Phi\beta_{2s-1})\gamma_{s-2}\beta_{2s-2}\gamma_{2s-3}\ldots\ldots\beta_{1}\gamma_{0})^{\,i_{k_{2s}+1},i_{k_{2s}+2},\ldots,i_l}\\[.05in]
&=& ((\Phi^{i_{k_{2s-1}}, i_{k_{2s-1}+1}, \dots, i_{k_{2s}-1} , i_{k_{2s}}})\gamma_{s-2}\beta_{2s-2}\gamma_{2s-3}\ldots\ldots\beta_{1}\gamma_{0})^{\,i_{k_{2s}+1},i_{k_{2s}+2},\ldots,i_l}\\[.05in]
&=& (\Phi\gamma_{s-2}\beta_{2s-2}\gamma_{2s-3}\ldots\ldots\beta_{1}\gamma_{0})^{\,{i_{k_{2s-1}}, i_{k_{2s-1}+1}, \ldots, i_{k_{2s}-1},i_{k_{2s}}},i_{k_{2s}+1},i_{k_{2s}+2},\ldots,i_l}\\[.02in]
&\vdots&\\[.02in]
&=&\, \Phi^{i_1,\ldots,i_l} \,=\, \Psi \,=\, \Phi\psi.
\end{array} \]

Therefore, since the images of $\Phi$ under $\psi$ and $\widehat {\psi}$ coincide, we must have $\psi=\widehat{\psi}\in\langle\mathcal{G}_{I}\rangle$. This completes the proof.
\end{proof}
\smallskip

The elements of $\mathcal{G}_{I}$ are called the {\em distinguished generators} of $\Gamma(\mathcal{P})$ {\em with respect to\/}~$\Phi$. If there is no possibility of confusion, we omit the reference to the base flag and simply refer to the  elements of $\mathcal{G}_{I}$ as the distinguished generators of $\Gamma(\mathcal{P})$. There are three kinds of distinguished generators, to some extent overlapping. 

The generators of the {\it first\/} kind, $\rho_i$ ($i \in I$), and the generators of the {\it third\/} kind, $\alpha_{j,i,j}$ ($i \in I$, $j \notin I$), are involutions. This is clear for the generators $\rho_i$. The generators $\alpha_{j,i,j}$ are associated with the flag $\Phi^j$ in the same way as the generators $\rho_i$ are with $\Phi$, so these must be involutions as well. More explicitly, 
\[ (\Phi\alpha_{j,i,j})\alpha_{j,i,j}=(\Phi^{j,i,j})\alpha_{j,i,j}
=(\Phi\alpha_{j,i,j})^{j,i,j}=(\Phi^{j,i,j})^{j,i,j} = \Phi^{j,i,j,j,i,j}=\Phi,\]
showing that $\alpha_{j,i,j}^{\,2}=1$. Moreover,   
\begin{equation}
\label{alphrho}
\alpha_{j,i,j} = \rho_i\,\mbox{ if } | j - i | \geq 2,
\end{equation}
since in this case $\Phi^{j,i,j}=\Phi^i$. 

For the generators of the {\it second\/} kind, $\alpha_{j,k}$ ($j, k\notin I$), we have 
\begin{equation} 
\label{invalphajk}
\alpha^{-1}_{j,k}= \alpha_{k,j},
\end{equation}
since $\Phi^{k,j,j,k}=\Phi$ gives $\alpha_{j,k}\alpha_{k,j}=1$. We claim that $\alpha_{j,k}$ has period $p_{jk}$, where 
\[p_{jk}=
\left\{\begin{array}{ll}
2   &\mbox{if } |j-k|\geq 2,\\
p_k&\mbox{if } j=k-1,\\
p_j&\mbox{if } k=j-1.
\end{array}
\right.\]
Here $p_l$ is the $l^{th}$ entry in the top row of the double Schl\"afli symbol of $\mathcal{P}$, which by assumption is aligned with the orbit of $\Phi$. To see that $\alpha_{j,k}$ is an involution if $|j-k| \geq 2$, observe that $\Phi^{j,k}=\Phi^{k,j}$ in this case and thus $\alpha_{j,k}=\alpha_{k,j}$, giving $\alpha_{j,k}^{2}=1$. If $j=k-1$, the element $\alpha_{j,k}=\alpha_{k-1,k}$ fixes each face in $\Phi$ except the $(k-1)$-face and $k$-face, and cyclically permutes, or shifts one step along, the $p_k$ $k$-faces in the 2-section $\Phi_{k+1}/\Phi_{k-2}$ of $\mathcal{P}$ according as $p_k$ is finite or infinite. If $k=j-1$, then $\alpha_{j,k}$ has the same order as $\alpha_{k,j}=\alpha_{j-1,j}$, which is $p_j$.

Thus
\begin{eqnarray}
\label{rel0}
\rho_i^2=\alpha_{j,i,j}^2 = \alpha^{p_{jk}}_{j,k} = \alpha_{j,k}\alpha_{k,j}  = 1 \quad (i\in I,\ \! j,k\notin I).
\end{eqnarray}
Moreover, we also have the relations
\begin{eqnarray}
\label{relations2}
(\rho_{i}\rho_{l})^{p_{il}} = (\alpha_{j,i,j} \alpha_{j,l,j})^{q_{il}} = 1 \quad (i,l\in I,\ \! j \notin I),
\end{eqnarray}
where $p_{il}:=p_{li}:=q_{il}:=q_{li}:=2$ if $|i-l | \geq 2$, and $p_{il}:=p_{li}:=p_l$ and $q_{il}:=q_{li}:=q_l$ if $i=l-1$. These relations can again be verified by evaluating the pertinent elements on $\Phi$. For example, we have 
\[ (\Phi\alpha_{j,i,j})\alpha_{j,l,j}=(\Phi^{j,i,j})\alpha_{j,l,j}
=(\Phi\alpha_{j,l,j})^{j,i,j}=(\Phi^{j,l,j})^{j,i,j} = \Phi^{j,l,j,j,i,j}=\Phi^{j,l,i,j},\]
so if $|i - l | \geq 2$ then $\Phi^{j,l,i,j}=\Phi^{j,i,l,j}$ and hence $\alpha_{j,i,j}\alpha_{j,l,j}$ has period $2$; and if $i=l-1$ then $\alpha_{j,i,j}\alpha_{j,l,j}$ fixes every face of the $j$-adjacent flag $\Phi^j$ of $\Phi$ except the $(l-1)$-face and the $l$-face, and cyclically permutes the $q_l$ $l$-faces in the 2-section $\Phi^{j}_{l+1}/\Phi^{j}_{l-2}$ of $\mathcal{P}$ (note here that $\Phi^j$ and $\Phi$ are in distinct flag orbits since $j\notin I$).

There are further relationships between the generators of $\Gamma(\mathcal{P})$. In particular, 
\begin{eqnarray}
\label{rel0more}
\alpha_{j,k} \alpha_{j,i,j} \alpha_{k,j} = \alpha_{k,i,k} \quad (i\in I,\ j,k\notin I).
\end{eqnarray}
These relations can again be obtained by evaluating the elements on both sides on the base flag $\Phi$ and then observing that the results coincide:
\[\Phi\alpha_{j,k} \alpha_{j,i,j} \alpha_{k,j} = \Phi^{k,j,j,i,j,j,k}=\Phi^{k,i,k}=\Phi\alpha_{k,i,k}.\]

We note two relations that can be derived as special cases of (\ref{rel0more}). First, if $i\in I$ and $j,i+1\notin I$ such that $|j-i| \geq 2$, then
\begin{eqnarray}
\label{rel1}
\alpha_{j,i+1} \rho_{i} \alpha_{i+1,j} = \alpha_{i+1,i,i+1}.
\end{eqnarray}
In fact, $\rho_{i}=\alpha_{j,i,j}$ in this case, so (\ref{rel1}) follows from (\ref{rel0more}) with $k=i+1$. Second, if $i \in I$ and $j-1,j \notin I$ such that either $i < j-2$ or $i > j+1$, then
\begin{eqnarray}
\label{rel2}
\alpha_{j-1,j} \rho_i = \rho_i \alpha_{j-1,j}.
\end{eqnarray}
For the proof, apply (\ref{rel0more}) with $j$ and $k$ replaced by $j-1$ and $j$, respectively, and note that $\rho_{i}=\alpha_{j-1,i,j-1}$.
\smallskip

Observe also that if an entry $p_l$ in the top row of the double Schl\"afli symbol is odd, then either $l-1, l \in I$ or $l-1, l \notin I$. In fact, suppose to the contrary that, for example, $l-1 \in I$ and $l \notin I$. Then the involutions $\rho_{l-1}$ and $\alpha_{l, l-1, l}$ both leave the common faces of $\Phi$ and $\Phi^l$ of ranks distinct from $l-1$ and $l$ invariant, and in particular act on the section $\Phi_{l+1} / \Phi_{l-2}$ of $\mathcal{P}$ of rank 2. On this section, $\rho_{l-1}$ and $\alpha_{l, l-1, l}$ act like reflection symmetries of a $p_l$-gon in the perpendicular bisectors of adjacent edges. Hence, since $p_l$ is odd, the subgroup $\langle \rho_{l-1}, \alpha_{l, l-1, l} \rangle$ of $\Gamma(\mathcal{P})$ is isomorphic to the dihedral group $\mathcal{D}_{p_l}$ and must act flag transitively on the section $\Phi_{l+1} / \Phi_{l-2}$. It follows that this subgroup of $\Gamma(\mathcal{P})$ must also contain an automorphism of $\mathcal{P}$ mapping $\Phi$ to $\Phi^{l}$. Hence $l \in I$, which is a contradiction. 

Thus $p_l$ must be even if exactly one of $l-1$ and $l$ lies in~$I$. In this case arguments similar to those above show that 
\begin{equation}
\label{radih}
\begin{array}{rcl}
\langle\rho_{l-1},\alpha_{l,l-1,l}\rangle\!\! &\cong&\!\! D_{p_{l}/2} \;\mbox{ if } l-1\in I,\, l\notin I, \\[.04in]
\langle\rho_{l},\alpha_{l-1,l,l-1}\rangle\!\! &\cong&\!\! D_{p_{l}/2} \;\mbox{ if } l-1\notin I,\, l\in I .
\end{array}
\end{equation}
\smallskip

The distinguished generators for $\Gamma(\mathcal{P})$ depend on the choice of the base flag $\Phi$ of $\mathcal{P}$. If $\Psi$ is a flag in the same orbit as $\Phi$, then the generating set $\mathcal{G}_{I}(\Psi)$ is conjugate in $\Gamma(\mathcal{P})$ to the generating set $\mathcal{G}_{I}(\Phi)$, and the conjugation is by the element of $\Gamma(\mathcal{P})$ that maps $\Phi$ to $\Psi$. This is no longer true if the two flags are in distinct orbits. If $\Psi$ is in a different flag orbit than $\Phi$, then the generating set $\mathcal{G}_{I}(\Psi)$ is conjugate in $\Gamma(\mathcal{P})$ to the generating set associated with a flag adjacent to $\Phi$ but not in the same flag orbit as $\Phi$. Thus, in order to investigate the generating sets associated with flags of the other orbit it suffices to choose a $j$-adjacent flag of~$\Phi$ with $j\notin I$ as the base flag. 

Now suppose $\Psi$ is a flag adjacent to $\Phi$ but not in the same orbit as $\Phi$. Then $\Psi=\Phi^{j_0}$ for some $j_0 \notin I$ and the distinguished generators 
\[\rho_i',\;\alpha'_{j,k},\; \alpha'_{j,i,j}\quad (i \in I,\,j,k \notin I)\]
in the corresponding generating set $\mathcal{G}_{I}(\Psi)=\mathcal{G}_{I}(\Phi^{j_0})$ are related to those in $\mathcal{G}_{I}(\Phi)$ by the equations
\begin{eqnarray}
\label{changeofgenerators}
\rho'_i = \alpha_{j_0,i,j_0}, \qq \qq \qq \alpha'_{j,k}= \alpha_{k,j_0}\alpha_{j_0,j}, \qq \qq \qq \alpha'_{j,i,j}=\alpha_{j,j_0}\rho_i\alpha_{j_0,j}.
\end{eqnarray}
Once again, these equations can be verified by evaluating both sides of an equation on $\Phi$. The details are as follows.

First, the definition of $\rho_i'$ gives $(\Phi\rho'_i)^{j_0} =(\Phi^{j_0})\rho'_i = (\Phi^{j_0})^i$, hence 
\[ \Phi\rho'_i =\Phi^{j_0,i,j_0} = \Phi\alpha_{j_0,i,j_0}\]
and therefore $\rho'_i = \alpha_{j_0,i,j_0}$.
Similarly, for $\alpha'_{j,k}$ the defining property gives 
$(\Phi\alpha'_{j,k})^{j_0} = (\Phi^{j_0})\alpha'_{j,k} = (\Phi^{j_0})^{j,k}$,
hence 
\[ \Phi\alpha'_{j,k} =\Phi^{j_0,j,k,j_0} = (\Phi\alpha_{j_0,j})^{k,j_0} 
= (\Phi^{k,j_0})\alpha_{j_0,j} = \Phi\alpha_{k,j_0}\alpha_{j_0,j}\]
and therefore $\alpha'_{j,k} = \alpha_{k,j_0}\alpha_{j_0,j}$. Finally, from the definition of $\alpha'_{j,i,j}$ we obtain  
$(\Phi\alpha'_{j,i,j})^{j_0} =(\Phi^{j_0})\alpha'_{j,i,j} =(\Phi^{j_0})^{j,i,j}$,
hence  
\[\begin{array}{ll}
\Phi\alpha'_{j,i,j} = \Phi^{j_0,j,i,j,j_0} &= (\Phi\alpha_{j_0,j})^{i,j,j_0} \\[.03in]
&= (\Phi^{i,j,j_0})\alpha_{j_0,j} \\[.03in]
&= ((\Phi\rho_{i})^{j,j_0})\alpha_{j_0,j} = (\Phi^{j,j_0})\rho_{i}\alpha_{j_0,j}
=\Phi\alpha_{j,j_0}\rho_i\alpha_{j_0,j}, 
\end{array}\]
leading to the final equation $\alpha'_{j,i,j}=\alpha_{j,j_0}\rho_i\alpha_{j_0,j}$.
Thus the generators from $\mathcal{G}_{I}(\Psi)=\mathcal{G}_{I}(\Phi^{j_0})$ can be expressed in a relatively simple manner in terms of the generators from $\mathcal{G}_{I}(\Phi)$ as described in (\ref{changeofgenerators}). 
\smallskip

Note also that the generators $\alpha_{j_0,i,j_0}=\rho'_{i}$ commute under certain conditions, just like the original generators $\rho_i$ do:
\begin{equation}
\label{alphascom}
\alpha_{j_0,i,j_0}\alpha_{j_0,l,j_0} = \alpha_{j_0,l,j_0}\alpha_{j_0,i,j_0} \qquad (i,l\in I,\, |i-l|\geq 2).
\end{equation} 
In fact, working out the effects of the two products on the base flag leads to the same results. More precisely, 
\[\Phi\alpha_{j_0,i,j_0}\alpha_{j_0,l,j_0}=
(\Phi^{j_0,i,j_0})\alpha_{j_0,l,j_0}=(\Phi\alpha_{j_0,l,j_0})^{j_0,i,j_0}
=\Phi^{j_0,l,j_0,j_0,i,j_0}=\Phi^{j_0,l,i,j_0},\]
and similarly, $\Phi\alpha_{j_0,l,j_0}\alpha_{j_0,i,j_0}=\Phi^{j_0,i,l,j_0}$;
since $|i-l|\geq 2$, these two flags coincide.

For example, if $|I|=n-1$ and $\bar{I}=\{j_0\}$ with $j_{0}\neq 0,n-1$, the relations in (\ref{alphascom}) include the special case 
\begin{equation}
\label{alphaj0com}
\alpha_{j_0,j_{0}-1,j_0}\alpha_{j_0,j_{0}+1,j_0} = \alpha_{j_0,j_{0}+1,j_0}\alpha_{j_0,j_{0}-1,j_0} .
\end{equation}

\subsection{Stabilizers}
\label{stabsection}

As before, let $\mathcal{P}$ be a two-orbit $n$-polytope in the class $2_I$ with base flag $\Phi$, and let $\mathcal{G}_{I}$ denote the distinguished generating set $\{\rho_{i},\alpha_{j,k},\alpha_{j,i,j}\}$ for $\Gamma(\mathcal{P})$ determined by the base flag $\Phi$. For simplicity we set 
\[\Gamma:=\Gamma(\mathcal{P}).\] 
In contexts where we view this group along with the alternative generating set $\{\rho'_{i},\alpha'_{j,k},\alpha'_{j,i,j}\}$ determined by a flag that is adjacent to $\Phi$ but from the other flag-orbit (as described at the end of the previous section), we sometimes denote $\Gamma$ by $\Gamma'$.
\smallskip

The specific nature of the generators in $\mathcal{G}_{I}$ permits us to describe the stabilizers of the subchains of $\Phi$ in $\Gamma$. Each subchain of $\Phi$ is of the form 
\[\Phi_J := \{ \Phi_j \mid j \in J \},\]
for some $J \subseteq N$. Let $\Gamma_{\Phi_J}$ denote the stabilizer of $\Phi_J$ in $\Gamma$. Clearly, for $J, K \subseteq N$ we have $\Phi_{J \cup K}=\Phi_J \cup \Phi_K$ and therefore
\begin{eqnarray}
\label{intersection1}
\Gamma_{\Phi_{J \cup K}} = \Gamma_{\Phi_J} \cap \Gamma_{\Phi_K}.
\end{eqnarray}

For $J \subseteq N$ define the {\it distinguished subgroups} $\Gamma_J$ of $\Gamma$ by 
\begin{equation}
\label{gamJ}
\Gamma_J:=\langle \rho_i, \alpha_{j,k},  \alpha_{j,i,j} \mid i \in I \cap \bar{J}; \, j,k \in \bar{I} \cap \bar{J} \rangle.
\end{equation} 
It is often convenient to relabel these subgroups using the complements of index sets. Accordingly, for $J \subseteq N$ we also define the subgroups $\Gamma^J$ by 
\begin{equation}
\label{gamJbar}
\Gamma^{J}:= \Gamma_{\bar{J}} = \langle \rho_i, \alpha_{j,k}, \alpha_{j,i,j} \mid i \in I \cap J; \, j,k \in \bar{I} \cap J \rangle.
\end{equation} 
Then, by definition, $\Gamma_N=\Gamma^{\emptyset}=1$, the trivial group, and $\Gamma_{\emptyset}=\Gamma^N=\Gamma$. 
\smallskip

The following lemma shows that the distinguished subgroups of $\Gamma$ are precisely the stabilizers of the subchains of the base flag $\Phi$.

\begin{lemma}
\label{gamstab}
For each $J\subseteq N$ we have $\Gamma_{\Phi_J}=\Gamma_J=\Gamma^{\bar{J}}$.
\end{lemma}

\begin{proof}
Let $J\subseteq N$. First note that the generators of $\Gamma_J$ stabilize the faces in $\Phi_J$, so clearly $\Gamma_J$ lies in $\Gamma_{\Phi_J}$. For the converse we can adapt the proof of Theorem~\ref{generators} as follows. 

Suppose $\psi\in\Gamma_{\Phi_J}$. Now, if the element $\psi$ in the proof of Theorem~\ref{generators} lies in $\Gamma_{\Phi_J}$, as is the case here, then the strong flag connectedness of $\mathcal{P}$ shows that we may take the flags in (\ref{adjacentsequence1}) in such a way that $i_1, \dots, i_l \notin J$. Therefore, if again $\psi$ is expressed as $\widehat{\psi}$, and then $\widehat{\psi}$ is written as in (\ref{star}), then none of the subscripts occurring in the expressions for the terms $\beta_q$ and $\gamma_q$ lies in $J$. Hence, $\psi=\widehat{\psi} \in \Gamma_J$. Thus $\Gamma_{\Phi_J}$ also lies in $\Gamma_J$.
\end{proof}

The collection of distinguished subgroups of $\Gamma$ behaves nicely with respect to taking intersections and, in particular, has the following property called the {\it intersection property} of~$\Gamma$. This property can be expressed in one of two equivalent ways, depending on whether the groups are indexed by subscripts or superscripts.

\begin{theorem}
\label{intcondlemma}
The collections of distinguished subgroups $\big\{\Gamma_{J}\!\,\big\}_{\!J\subseteq N}$ and $\big\{\Gamma^{J}\!\,\big\}_{\!J\subseteq N}$ of $\Gamma$ satisfy the following properties:
\begin{eqnarray}
\label{intcond}
\begin{array}{ll}
\Gamma_{J} \cap \Gamma_{K} = \Gamma_{J\cup K} &\quad (J,K \subseteq N),\\[.1in]
\Gamma^{J} \cap \Gamma^{K} = \Gamma^{J\cap K} &\quad (J,K \subseteq N).
\end{array}
\end{eqnarray}
\end{theorem}

\begin{proof}
The first property follows immediately from (\ref{intersection1}) and Lemma~\ref{gamstab}. In fact, 
\[ \Gamma_{J} \cap \Gamma_{K} = \Gamma_{\Phi_J} \cap \Gamma_{\Phi_K}
= \Gamma_{\Phi_{J\cup K}}= \Gamma_{J\cup K}.\]
The second property just rephrases the first. In fact, $\Gamma^{J}=\Gamma_{\bar{J}}$,\  $\Gamma^{K}=\Gamma_{\bar{K}}$, and therefore
\[\Gamma^{J\cap K}=\,\Gamma_{\bar{J\cap K}}\,=\,\Gamma_{\bar{J}\cup \bar{K}}
\,=\, \Gamma_{\bar{J}}\cap\Gamma_{\bar{K}} \,=\, \Gamma^{J}\cap\Gamma^{K}.\]
This completes the proof.
\end{proof}
\smallskip

Every section of $\mathcal{P}$ lying between two faces of the base flag naturally determines a subgroup $\Gamma_J$ acting on this section. The following lemma tells us when the action of this subgroup is flag-transitive.

\begin{lemma}
\label{oneorbitsections}
Let $\mathcal{P}$ be a two-orbit $n$-polytope in the class $2_I$, with $I \subset N$, let $\Phi$ be its base flag, and let $-1 \leq r \leq s \leq n$. Set $J:=N \setminus \{r+1, \dots, s-1\}$, so that $\bar{J}= \{r+1,\dots, s-1\}$. Then $\Gamma_J$ acts flag-transitively on the section $\Phi_s / \Phi_r$ of $\mathcal{P}$ if and only if $\bar{J} \subseteq I$. In this case $\Phi_s / \Phi_r$ is a regular $n$-polytope with automorphism group isomorphic to $\Gamma_J$.
\end{lemma}

\begin{proof}
If $\bar{J} \subseteq I$, then $\Gamma_J$ is generated by $\rho_{r+1}, \dots, \rho_{s-1}$. Hence, since $\mathcal{P}$ is strongly flag-connected, the group $\Gamma_J$ acts flag transitively on $\Phi_s / \Phi_r$; in particular, $\Phi_s / \Phi_r$~is a regular $n$-polytope with group $\Gamma_J$.

Conversely, if $\Gamma_J$ acts flag transitively on $\Phi_s/\Phi_r$, then for each $l\in\bar{J}$ there exists an element $\widehat{\rho}_l$ of $\Gamma_J$ that takes the flag $\{\Phi_{r+1}, \dots, \Phi_{s-1}\}$ of $\Phi_s / \Phi_r$ to its $l$-adjacent flag in $\Phi_s / \Phi_r$. Since the elements of $\Gamma_J$ also fix the faces $\Phi_k$ with $k\in J$, the element $\widehat{\rho_l}$ must take the base flag $\Phi$ to its $l$-adjacent flag $\Phi^l$ in $\mathcal{P}$. Thus $l\in I$, by the definition of $I$. It follows that $\bar{J} \subseteq I$. 
\end{proof}
\smallskip

We also require the stabilizers of subchains of those flags in $\Gamma\,(=\Gamma')$ which are adjacent to~$\Phi$ but not from the same orbit as $\Phi$. Any such flag, $\Phi^{j_0}$ with $j_0 \notin I$ (say), determines distinguished generators $\rho_i'$, $\alpha'_{j,k}$, $\alpha'_{j,i,j}$ of $\Gamma$ as in (\ref{changeofgenerators}). For $J \subseteq N$ define the subgroups $\Gamma'_J$ of $\Gamma$ by 
\begin{equation}
\label{gamjprime}
\Gamma'_J:=
\langle \rho'_i, \alpha'_{j,k},  \alpha'_{j,i,j} \mid i \in I \cap \bar{J}; \, j,k \in \bar{I} \cap \bar{J} \rangle.
\end{equation} 
In particular, $\Gamma'_N$ is the trivial group and $\Gamma'_{\emptyset}=\Gamma'=\Gamma$.  For the sake of completeness we also define the subgroups $(\Gamma')^J$, for $J \subseteq N$, by $(\Gamma')^J:=\Gamma'_{\bar{J}}$. For $J\subseteq N$, let $\Phi^{j_0}_J$ denote the subchain of $\Phi^{j_0}$ given by 
\[\Phi^{j_0}_J:=\{(\Phi^{j_0})_{j}\mid j\in J\},\] 
and let $\Gamma_{\Phi^{j_0}_J}$ denote its stabilizer in $\Gamma$. 

Then note an immediate consequence of Lemma~\ref{gamstab}, which here is applied with $\Phi^{j_0}$ as the base flag. 

\begin{lemma}
\label{gamprimestabs}
For each $J\subseteq N$ we have $\Gamma_{\Phi^{j_0}_J}=\Gamma_{J}'=(\Gamma')^{\bar{J}}$.
\end{lemma}

Further, note the following immediate consequence of Theorem~\ref{intcondlemma}, which also is applied with $\Phi^{j_0}$ as the base flag. 

\begin{theorem}
\label{intcondlemmaprime}
The collections of distinguished subgroups $\big\{\Gamma'_{J}\!\,\big\}_{\!J\subseteq N}$ and $\big\{(\Gamma')^{J}\!\,\big\}_{\!J\subseteq N}$ of $\Gamma'=\Gamma$ satisfy the following properties:
\begin{eqnarray}
\label{intcondprime}
\begin{array}{ll}
\Gamma_{J}' \cap \Gamma_{K}' = \Gamma_{J\cup K}' &\quad (J,K \subseteq N),\\[.1in]
(\Gamma')^{J} \cap (\Gamma')^{K} = (\Gamma')^{J\cap K} &\quad (J,K \subseteq N).
\end{array}
\end{eqnarray}
\end{theorem}
\bigskip

The subgroups indexed by one-element subsets $J$ turn out to be particularly important. For each $l\in N$, we set 
\begin{equation}
\label{gamlprime}
\Gamma_l := \Gamma_{\{l\}},\;\; \Gamma^l:=\Gamma^{\{l\}},\;\; 
\Gamma'_l := \Gamma'_{\{l\}},\;\; (\Gamma')^l:=(\Gamma')^{\{l\}}.
\end{equation}
Then $\Gamma_l$ is just the stabilizer of the $l$-face $\Phi_l$ in the base flag $\Phi$ of $\mathcal{P}$, and
\begin{equation}
\label{gamsubl}
\Gamma_l =  \langle \rho_i,\, \alpha_{j,k},\, \alpha_{j,i,j}  \mid i \in I, \, j,k \notin I, \,  i, j, k \neq l \rangle. 
\end{equation}
Similarly, $\Gamma'_{l}$ is the stabilizer of the $l$-face $\Phi_{l}^{j_0}$ in the $j_{0}$-adjacent flag $\Phi^{j_0}$ of $\Phi$, and 
\begin{equation}
\label{gamsublprime}
\Gamma'_l =  \langle \rho'_i,\, \alpha'_{j,k},\, \alpha'_{j,i,j}  \mid i \in I, \, j,k \notin I, \,  i, j, k \neq l \rangle. 
\end{equation}
Further, $\Gamma^l=\langle\rho_l\rangle$ if $l\in I$, and $\Gamma^l = \{1\}$ if $l\notin I$; and similarly, $(\Gamma')^l=\langle\rho'_l\rangle$ if $l\in I$, and $(\Gamma')^l = \{1\}$ if $l\notin I$.
\medskip

The two collections of subgroups $\big\{\Gamma_{J}\!\,\big\}_{\!J\subset N}$ and $\big\{\Gamma'_{J}\!\,\big\}_{\!J\subset N}$ of $\Gamma$ are intertwined as follows. For $J\subset N$ with $j_{0}\in J$, set 
\[J_{0}:=J\setminus\{j_0\}.\]
Then the following lemma holds.

\begin{lemma}
\label{intertwine}
Let $J\subset N$. Then, 
\[\Gamma_{J}'= \left\{
\begin{array}{ll}
\Gamma_{J} & \mbox{if }\, j_{0}\not\in J\\
\Gamma_{J_0}\cap \Gamma_{j_0}' & \mbox{if }\, j_{0}\in J.
\end{array}\right.\]
\end{lemma}

\begin{proof}
If $j_{0}\not\in J$, then $\Phi^{j_0}_J=\Phi_J$, so Lemmas~\ref{gamstab} and~\ref{gamprimestabs} give 
\[\Gamma_{J}'=\Gamma_{\Phi^{j_0}_J}=\Gamma_{\Phi_J}=\Gamma_J.\]
This proves the first equation. On the other hand, if $j_{0}\in J$, then $J=J_{0}\cup\{j_0\}$ and therefore, by Theorem~\ref{intcondlemmaprime} and the first equation applied to $J_0$,
\[\Gamma_{J}'\,=\,\Gamma'_{J_{0}\cup\{j_0\}}\,=\,\Gamma'_{J_0}\cap\Gamma'_{\{j_0\}}
\,=\, \Gamma_{J_0}\cap\Gamma_{j_0}'.\]
This also proves the second equation.
\end{proof}
\smallskip

The following lemma says that every section of a two-orbit polytope in the class $2_I$ is either regular or is itself a two-orbit polytope in a class determined by $I$ and the set of ranks of proper faces in the section. 

Recall that $\mathcal{S}_{r,s}(\mathcal{P})$ denotes the set of all sections $G/F$ of $\mathcal{P}$ that are determined by an incident pair of an $r$-face $F$ and an $s$-face $G$ of $\mathcal{P}$. We also require the following notation. For $L\subseteq \mathbb{Z}$ and $m\in\mathbb{Z}$, let $L-m:=\{l-m\mid l\in L\}$.

\begin{lemma}
\label{sect2orbit}
Let $\mathcal{P}$ be a two-orbit $n$-polytope in the class $2_I$, with $I\subset N$, and let $-1\leq r\leq s\leq n$ and $\mathcal{Q}$ be a section in $\mathcal{S}_{r,s}(\mathcal{P})$. Then $\mathcal{Q}$ is either a regular $(s-r-1)$-polytope or a two-orbit $(s-r-1)$-polytope in the class $2_{I'}$ where 
\[I':=(I\cap \{r+1,\ldots,s-1\})-(r+1).\]
\end{lemma}

\begin{proof}
Let $\Phi$ be any flag of $\mathcal{P}$. By Theorem~\ref{sections} and its proof, $\mathcal{Q}$ is equivalent under $\Gamma(\mathcal{P})$ to $\Phi_{s}/\Phi_{r}$ if $\bar{I}\not\subseteq\{r,s\}$; or to $\Phi_{s}/\Phi_{r}$ or $\Phi_{s}^{k}/\Phi_{r}^{k}$, with any $k\in\bar{I}$, if $\bar{I}\subseteq\{r,s\}$. Hence we may assume that $\mathcal{Q}$ coincides with $\Phi_{s}/\Phi_{r}$ or $\Phi_{s}^{k}/\Phi_{r}^{k}$, respectively. Set $J:=N\setminus\{r+1,\ldots,s-1\}$, so that $\bar{J}=\{r+1,\ldots,s-1\}$. As before, $\Phi_{J}=\{\Phi_{l}\mid l \in J\}$.

First suppose that $\mathcal{Q}=\Phi_s/\Phi_r$. Then $\Gamma_{\Phi_J}$ is a subgroup of $\Gamma(\mathcal{Q})$, and by Lemma~4, $\Gamma_{\Phi_J}=\Gamma^{\bar{J}}=\Gamma_J$. We show that $\mathcal{Q}$ has at most two flag orbits under $\Gamma_{\Phi_J}$. Now suppose $\Psi'$ is any flag of $\mathcal{Q}$, and set $\Psi:=\Psi'\cup\Phi_J$. Two possibilities can occur. 

Suppose $\bar{I}\cap\bar{J}\neq\emptyset$, and let $t\in\bar{I}\cap\bar{J}$. Then $\Phi_J$ is a subset of all three flags $\Psi$, $\Phi$ and $\Phi^t$, and hence $\Psi$ is equivalent to exactly one of $\Phi$ or $\Phi^t$ under an automorphism of $\mathcal{P}$ that necessarily must fix each face in $\Phi_J$ and therefore lie in $\Gamma_{\Phi_J}$. Thus $\mathcal{Q}$ has at most two flag orbits under $\Gamma_{\Phi_J}$. In particular, if $\mathcal{Q}$ is not regular, then $\mathcal{Q}$ must have two flag orbits under $\Gamma_{\Phi_J}$, and $\mathcal{Q}$ is a two-orbit polytope with automorphism group $\Gamma_{\Phi_J}=\Gamma_{J}$. In this case, since the class of a two-orbit polytope is uniquely characterized by the set of generators of the form $\rho_i$ in its group, which here is given by $\{\rho_{i}\mid i\in I\cap\bar{J}\}$, the class type set of $\mathcal{Q}$ (relative to the rank function of $\mathcal{Q}$ inherited from $\mathcal{P}$) must coincide with $I\cap\bar{J}$. If $\mathcal{Q}$ is taken as an $(r-s-1)$-polytope in its own right, independent of $\mathcal{P}$, then the correct class type set for $\mathcal{Q}$ is $I':=(I\cap\bar{J})-(r+1)$.  (Note that, in principle, $\mathcal{Q}$ could still be regular if $\bar{I}\cap\bar{J}\neq\emptyset$. In this case $\Gamma_{\Phi_J}$ could only be a subgroup of index 2 in $\Gamma(\mathcal{Q})$, since $t\notin I$ and thus $\Gamma_{\Phi_J}$ cannot contain an automorphism mapping $\Phi$ to $\Phi^t$.)

On the other hand, if $\bar{I}\cap\bar{J}=\emptyset$, then $\bar{J}\subseteq I$ and hence $\Psi$ is equivalent to $\Phi$ under an automorphism in $\Gamma_{\Phi_J}$. In this case $\mathcal{Q}$ is regular. 

Now suppose $\bar{I}\subseteq\{r,s\}$ and $\mathcal{Q}=\Phi_s^k/\Phi_r^k$, where  $k\in\bar{I}$. We now employ the subgroups $\Gamma_{J}'=(\Gamma')^{\bar{J}}$ defined relative to the $k$-adjacent flag $\Phi^k$ of $\Phi$. As before, let $\Psi'$ be a flag of $\mathcal{Q}$ and set $\Psi:=\Psi'\cup\Phi_J^k$. Then $\Phi_J^k$ lies in both $\Psi$ and $\Phi^k$, and $\bar{J}\subseteq I$ since $\bar{I}\subseteq\{r,s\}$. Now apply the analogue of Lemma~\ref{rhoex} for the subgroups defined with respect to $\Phi^k$ (that is, for $\Gamma_{J}'$ and $\Gamma_{\Phi^k_J}$), as well as Lemma~\ref{gamprimestabs}. It follows that $\Psi$ can be mapped under $\Gamma_{J}'=\Gamma_{\Phi^k_J}$ to the flag $\Phi^k$. Hence $\mathcal{Q}$ is again regular in this case.
\end{proof}

Next we analyze in greater detail the structure of $\Gamma_l$, the stabilizer of the $l$-face in $\Phi$. To this end, we define the three subgroups $\Gamma_l^-$, $\Gamma_l^+$ and $\Gamma_l^{\pm}$ of $\Gamma_l$ as follows:
\begin{equation}
\label{Gamma-}
\begin{array}{lllll}
\Gamma_l^{-} &\!\!\!:=\!\!\!& \langle \rho_i,\, \alpha_{j,k},\, \alpha_{j,i,j}  \mid i \in I, \, j,k \notin I, \,  i, j, k <l \rangle 
&\!\!\!=\!\!\!& \Gamma_{\{l \dots, n-1\}},  \\[0.08in]
\Gamma_l^+ &\!\!\!:=\!\!\!& \langle \rho_i,\, \alpha_{j,k}, \,\alpha_{j,i,j}  \mid i \in I,\, j,k \notin I, \,  i, j, k > l \rangle
&\!\!\!=\!\!\!& \Gamma_{\{0 \dots, l\}}, \\[0.08in]
\Gamma_l^{\pm} &\!\!\!:=\!\!\!& \langle\alpha_{j,k}\mid j,k \notin I, \,  j<l<k \rangle. 
\end{array}
\end{equation}
Note that the generators $\alpha_{j,k}$ of $\Gamma_l^{\pm}$ are all involutions, by the remarks following (\ref{invalphajk}). Moreover, for each $l\in N$ the two subgroups $\Gamma_l^{-}$ and $\Gamma_l^+$ centralize each other; that is, each generator of $\Gamma_l^{-}$ commutes with each generator of $\Gamma_l^+$. This follows from the fact that the elements of $\Gamma_l^-$ fix the $r$-faces $\Phi_r$ of $\Phi$ of rank $r\geq l$, while the elements of $\Gamma_l^+$ fix the $r$-faces $\Phi_r$ of $\Phi$ of rank $r\leq l$; bear in mind that an automorphism of a polytope which fixes a flag must necessarily be the identity. Further, since $\Gamma_l^-$ and $\Gamma_l^+$ commute at the level of elements and have trivial intersection, the product of groups $\Gamma_l^{-}\Gamma_l^{+} = \Gamma_l^{+} \Gamma_l^{-}$ must be a subgroup of $\Gamma$ isomorphic to the direct product $\Gamma_l^{-}\times\Gamma_l^{+}$. Note that this subgroup must contain the product of any two generators of $\Gamma_l^{\pm}$. More specifically, 
\begin{eqnarray}
\label{alphcom}
\alpha_{j,k}\,\alpha_{j',k'} = \alpha_{k',k}\,\alpha_{j',j}=\alpha_{j',j}\,\alpha_{k',k} \in \Gamma_l^{-}\Gamma_l^{+} ,
\end{eqnarray}
where $j,j',k,k'\notin I$ with $j,j'< l <k,k'$. The equality between the three products in (\ref{alphcom}) can be verified as in similar situations before, by computing the image of $\Phi$ under each product and then noting that 
\[ \Phi^{j',k',j,k}=\Phi^{j',j,k',k}=\Phi^{k',k,j',j}.\]
\smallskip

The following lemma describes the structure of the subgroups $\Gamma_l$ of $\Gamma$.

\begin {lemma}
\label{gasub}
The subgroup $\Gamma_l^{-} \Gamma_l^{+}$ of $\Gamma_l$ is isomorphic to $\Gamma_l^{-}\times\Gamma_l^{+}$ and has index at most 2 in $\Gamma_l$. In particular, the index is 2 if and only if there exist $j,k \notin I$ such that $j<l<k$ (that is, if and only if the subgroup $\Gamma_l^{\pm}$ is nontrivial). In this case, for any such $j$ and $k$, 
$$\Gamma_l=\Gamma_l^{-}\Gamma_l^{+}\langle \alpha_{j,k} \rangle \cong \Gamma_l^{-}\Gamma_l^{+} \ltimes C_2 
\cong (\Gamma_l^{-}\times\Gamma_l^{+})\ltimes C_2.$$
If the index is $1$, then, of course,
\[\Gamma_l=\Gamma_l^{-}\Gamma_l^{+} \cong \Gamma_l^{-}\times\Gamma_l^{+}.\]
\end{lemma}

\begin{proof}
All the generators of $\Gamma_l$, except for $\alpha_{j,k}$ with $j<l<k$, lie in $\Gamma_l^{-}\Gamma_l^{+}$. Note here that $\alpha_{j,i,j}=\rho_i$ if $i<l<j$ or $j<l<i$. The excluded generators $\alpha_{j,k}$ with $j<l<k$ are involutions, and so $\alpha_{j,k}^{-1}=\alpha_{j,k}$. If $j,j', k,k' \notin I$ such that $j,j' < l < k,k'$, then (\ref{alphcom}) shows that $\alpha_{j,k} \in \Gamma_l^{-}\Gamma_l^{+}\alpha_{j',k'}$; hence the two cosets $\Gamma_l^{-}\Gamma_l^{+}\alpha_{j,k}$ and $\Gamma_l^{-}\Gamma_l^{+}\alpha_{j',k'}$ must be the same. Now the  lemma follows.
\end{proof}

Note further that the generators $\alpha_{j,k}$ of $\Gamma_l^{\pm}$, with $j,k\notin I$ and $j<l<k$, normalize each of the subgroups $\Gamma_l^{-}$ and $\Gamma_l^{+}$, that is,
\begin{eqnarray}
\alpha_{j,k}^{-1}\Gamma_l^{-} \alpha_{j,k} = \Gamma_l^{-}, \;\qq \alpha_{j,k}^{-1}\Gamma_l^{+} \alpha_{j,k} = \Gamma_l^{+}.
\end{eqnarray}
More explicity, under conjugation by $\alpha_{j,k}$, the generators of $\Gamma_l^{-}$ are transformed as follows. If $i\in I$ and $s,t \notin I$ such that $i,s,t<l$, then 
\begin{eqnarray} 
\label{asrs}
\alpha_{j,k}^{-1}\rho_i\alpha_{j,k}=\alpha_{j,i,j}, \ \
\alpha_{j,k}^{-1}\alpha_{s,t}\alpha_{j,k}= \alpha_{t,j}\alpha_{j,s}, \ \
\alpha_{j,k}^{-1}\alpha_{s,i,s}\alpha_{j,k}=\alpha_{s,j}\rho_i\alpha_{j,s}.
\end{eqnarray}
Bear in mind here that $k>l$. Once again, this can be verified by evaluating each side of an equation in (\ref{asrs}) on the base flag $\Phi$. Similar relationships also hold for conjugation of the generators of the subgroup $\Gamma_l^{+}$ by $\alpha_{k,j}$, the inverse of $\alpha_{j,k}$. More precisely, if $i\in I$ and $s,t \notin I$ such that $i,s,t>l$, then 
\begin{eqnarray} 
\label{asrsinv}
\alpha_{k,j}^{-1}\rho_i\alpha_{k,j}=\alpha_{k,i,k}, \ \
\alpha_{k,j}^{-1}\alpha_{s,t}\alpha_{k,j}= \alpha_{t,k}\alpha_{k,s}, \ \
\alpha_{k,j}^{-1}\alpha_{s,i,s}\alpha_{k,j}=\alpha_{s,k}\rho_i\alpha_{k,s}.
\end{eqnarray}

\subsection{Characterizing the partial order}
\label{sec:order}

The goal of this section is to characterize the partial order of two-orbit polytopes in terms of the distinguished generators of the automorphism group, as summarized in Theorem~\ref{ordercharacterize} at the end of this section. For a regular or chiral polytope, the corresponding characterization involves intersections of cosets of face stabilizers and ultimately rests on the fact that the polytope is fully-transitive; that is, its automorphism group acts transitively on the faces of each rank. As we shall see, the situation is more complicated for arbitrary two-orbit polytopes, although the principal approach is similar. 

As before, let $\mathcal{P}$ be a two-orbit $n$-polytope in the class $2_I$, with $I\subset N=\{0,\ldots,n-1\}$, and let $\Phi=\{\Phi_0,\ldots,\Phi_{n-1}\}$ be its base flag. Recall from Theorem~\ref{almost-fully-transitive} that $\mathcal{P}$ is fully-transitive if and only if 
${\rm rd}(I)>1$; and that, if ${\rm rd}(I)=1$ and $I = N\setminus \{j_0\}$ for some $j_0\in N$, the polytope $\mathcal{P}$ has two orbits on the set of $j$-faces but still is $i$-face transitive for every $i\neq j_0$. Further, recall the results of Theorem~\ref{sections} about the transitivity properties of $\Gamma(\mathcal{P})$ on the sets of comparable sections of $\mathcal{P}$. In the present context, these results on sections are relevant because pairs of incident faces of $\mathcal{P}$ determine sections of $\mathcal{P}$ and vice versa.

In light of Theorem~\ref{almost-fully-transitive} and Theorem~\ref{sections}, the subsequent discussion falls naturally into three cases for the class type set $I$ given by the rank deficiency values 
\begin{equation}
\label{classcases}
{\rm rd}(I)\geq 3,\; {\rm rd}(I)=2, \mbox{ or }\,{\rm rd}(I)=1,
\end{equation}
or more explicitly, $|I| \leq n-3$, $|I| = n-2$, or $|I|=n-1$, respectively.

If ${\rm rd}(I)\geq 2$, as holds true in the first two cases of (\ref{classcases}), the polytope $\mathcal{P}$ is fully-transitive, so each face of $\mathcal{P}$ is equivalent under $\Gamma(\mathcal{P})$ to a face in the base flag $\Phi$. Moreover, if additionally ${\rm rd}(I)\geq 3$, then any two comparable sections of $\mathcal{P}$ are equivalent under $\Gamma(\mathcal{P})$.
However, if ${\rm rd}(I)=2$ and $\bar{I}=\{j_0,k_0\}$) for some $j_0,k_0 \in N$ with $j_{0}<k_0$, then $\Gamma(\mathcal{P})$ no longer acts transitively on each set of comparable sections, though $\mathcal{P}$ is still fully-transitive; in fact, in this case there are precisely two orbits on the set of sections determined by an incident pair of a $j_0$-face and a $k_0$-face. On the other hand, if ${\rm rd}(I)=1$, then the polytope $\mathcal{P}$ is not even fully-transitive.
\medskip

Our further analysis breaks down into three lemmas, which, when combined, establish Theorem~\ref{ordercharacterize} below. 
\smallskip

The first lemma characterizes incidence between pairs of faces which are in the same orbit as faces of $\Phi$. In particular, this provides a complete characterization of the partial order when ${\rm rd}(I)\geq 3$. The lemma will also describe the incidence between faces for most pairs of face ranks when ${\rm rd}(I)=2$ or ${\rm rd}(I)=1$. However, when ${\rm rd}(I)=1$, not all faces are in the same orbit as faces of $\Phi$ and special considerations are needed. In this case, if $\bar{I}=\{j_0\}$, there are two orbits of $j_0$-faces represented by $\Phi_{j_0}$ and $\Phi_{j_0}^{j_0}$, respectively, but the faces of any rank $i\neq j_0$ are in the same orbit as faces in $\Phi$. Our Lemma~\ref{incnminus1} below will describe incidence of pairs of faces in which one face is a $j_0$-face in the orbit of $\Phi_{j_0}^{j_0}$.

We begin with the most prevalent scenario for pairs of faces.

\begin{lemma}
\label{char1}
Let $\mathcal{P}$ be a two-orbit $n$-polytope in the class $2_I$, with $I \subset N$, and let $\Phi$ be the base flag of $\mathcal{P}$. Let $i,j$ be such $-1 \leq i \leq j \leq n$ and $\bar{I}\neq \{i,j\}$, and let $\phi, \psi \in \Gamma(\mathcal{P})$. Then the following equivalence holds:
\begin{eqnarray}
\label{equivnminus3}
\Phi_i\phi \leq\Phi_j\psi \;\Longleftrightarrow\; \Gamma_i\phi \cap \Gamma_j \psi \neq \emptyset. 
\end{eqnarray}
In particular this is true whenever ${\rm rd}(I)\geq 3$ or ${\rm rd}(I)=1$ (but note that in the latter case, the $i$-faces or $j$-faces of $\mathcal{P}$ may not all lie in the same orbit as $\Phi_i$ or $\Phi_j$, respectively).
\end{lemma}

\begin{proof}
One direction is straightforward. If $\Gamma_i\phi\cap\Gamma_j\psi\neq\emptyset$ and $\alpha\in\Gamma_i\phi\cap \Gamma_j\psi$, then
\[ \Phi_{i}\varphi=\Phi_{i}\alpha\leq\Phi_{j}\alpha=\Phi_{j}\psi,\]
and we are done. For the converse we treat the two cases $I \cup \{i,j\} \neq N$ and $I \cup \{i,j\} =N$ separately.

First, let $I\cup\{i,j\}\neq N$ and suppose that $\Phi_i\phi\leq\Phi_j\psi$. Choose $k\notin I\cup\{i,j\}$. If $\Psi$ is any flag of $\mathcal{P}$ such that $\Phi_i\phi, \Phi_j\psi \in \Psi$, then also $\Phi_i\phi, \Phi_j\psi \in \Psi^k$ and either $\Psi$ or $\Psi^k$ are in the same orbit as $\Phi$. By interchanging $\Psi$ or $\Psi^k$ if need be, we may assume that $\Psi$ itself is in the same orbit as $\Phi$, so $\Psi=\Phi\alpha$ for some $\alpha \in \Gamma(\mathcal{P})$. But then $\Phi_i\phi =  \Phi_i\alpha$ and $\Phi_j\psi=\Phi_j\alpha$, and hence $\alpha \in \Gamma_i\phi \cap \Gamma_j \psi$. Thus, $\Gamma_i\phi \cap \Gamma_j\neq\emptyset$.

Now let $I\cup\{i,j\}=N$. Then, since $\bar{I}\neq\{i,j\}$ by assumption, we have $\bar{I}=\{i\}$ or $\bar{I}=\{j\}$. First, suppose $\bar{I}=\{j\}$. By the definition of the class $2_I$ and by Lemma~\ref{rhoex} (applied with $K=\{j\}$), all flags containing $\Phi_j$ are in the same orbit under $\Gamma$ (in fact, even under $\Gamma_j$) as $\Phi$, and therefore all flags containing $\Phi_j\psi$ are in the same orbit under $\Gamma$ as $\Phi$. Hence, if $\Phi_i\phi \leq \Phi_j\psi$, then each flag containing both $\Phi_i\phi$ and $\Phi_j\psi$ lies  in the same orbit as $\Phi$. It follows that there exists $\alpha \in \Gamma(\mathcal{P})$ such that $\Phi_i\phi =  \Phi_i\alpha$ and $\Phi_j\psi=\Phi_j\alpha$, the latter showing that  $\alpha\in\Gamma_i\phi\cap\Gamma_j\psi$. Thus again, $\Gamma_i\phi \cap \Gamma_j\neq\emptyset$. The second case when $\bar{I}=\{i\}$ can be dealt with similarly. 
\end{proof}

The characterization of incidence in Lemma~\ref{char1} also provides a description of equality of faces in most cases. Clearly, two faces can only coincide if their ranks are the same, so this is the case $j=i$.  Thus Lemma~\ref{char1} will also tell us when two faces coincide, except when ${\rm rd}(I)=1$, $\bar{I}=\{j_0\}$, and the rank of the faces is $j_0$. In particular, when ${\rm rd}(I)\geq 2$, Lemma~\ref{char1} provides a full description of equality of faces of any rank.
\medskip

The next two lemmas deal with the incidence of faces for all the pairs of face ranks which are not covered by Lemma~\ref{char1}. These special considerations only involve class type sets $I$ with ${\rm rd}(I)\leq 2$. 
\smallskip

We start with the case ${\rm rd}(I)=2$. Here we may restrict ourselves to incident pairs of faces of distinct ranks $i$ and $j$ with $i<j$, as equality of faces has already been  covered by Lemma~\ref{char1}. 
\smallskip

\begin{lemma}
\label{char2}
Let $\mathcal{P}$ be a two-orbit $n$-polytope in the class $2_I$, with ${\rm rd}(I)=2$, and let $\Phi$ be the base flag of $\mathcal{P}$. Let $i,j$ be such that $-1 \leq i < j \leq n$ and $\bar{I}=\{i,j\}$, and let $\phi, \psi \in \Gamma(\mathcal{P})$.  Then the following equivalence holds: 
\begin{eqnarray}
\label{equivalence_n-2}
\Phi_i\phi \leq \Phi_j\psi\,\Longleftrightarrow\,\Gamma_i\phi \cap (\Gamma_j \cup \alpha_{i,j}\Gamma_j ) \psi \neq \emptyset.
\end{eqnarray}
\end{lemma}

\begin{proof} 
Define $\Sigma$ denote the intersection on the right hand side of (\ref{equivalence_n-2}).
Since the left cosets $\Gamma_j$ and $\alpha_{i,j}\Gamma_j$ of $\Gamma_j$ are disjoint, $(\Gamma_j \cup \alpha_{i,j}\Gamma_j )\psi$ is the disjoint union of $\Gamma_j\psi$ and $\alpha_{i,j}\Gamma_j\psi$. Moreover, 
\[\Sigma\,=\,(\Gamma_i\phi \cap \Gamma_j\psi) \cup (\Gamma_i\varphi \cap \alpha_{i,j}\Gamma_j \psi),\] 
which is also a disjoint union.

First, suppose that the condition on the right hand side of (\ref{equivalence_n-2}) holds;  that is, $\Sigma\neq\emptyset$. Let $\alpha \in \Sigma$. Then there are two possibilities. If $\alpha\in\Gamma_i\phi\cap\Gamma_j\psi$, then $\Gamma_{i}\phi=\Gamma_{i}\alpha$ and $\Gamma_{j}\psi=\Gamma_{j}\alpha$ and therefore
\[\Phi_i\phi =\Phi_{i}\alpha \leq \Phi_{j}\alpha=\Phi_j\psi.\] 
On the other hand, if  $\alpha \in \Gamma_i\phi\cap\alpha_{i,j}\Gamma_j  \psi$, then similarly $\Phi_i \phi = \Phi_i \alpha$ and $\Phi_{j}\psi =  \Phi_{j}\alpha_{i,j}^{-1}\alpha = \Phi_{j}\alpha_{j,i}\alpha$. Hence,
\[\Phi_i\,\phi = \Phi_i\, \alpha = \Phi_i^j \, \alpha \leq\Phi^j_j \,\alpha = \Phi^{j,i}_j \,\alpha = (\Phi \alpha_{j,i})_j\, \alpha = \Phi_j\, \alpha_{j,i} \,\alpha = \Phi_j \,\psi .\]
Thus in both cases, $\Phi_i\phi \leq \Phi_j\psi$, as required.

Conversely, suppose that $\Phi_i\phi \leq\Phi_j\psi$. By Lemma~\ref{rhoex} and the definition of the class $2_I$, with $\bar{I} = \{i,j\}$, any two flags containing the faces $\Phi_i\phi$ and $\Phi_j\psi$ must lie in the same orbit under $\Gamma$. More explicitly, by the strong flag-connectedness, any two such flags can be joined by a sequence of successively adjacent flags in which the adjacencies occur at ranks different from $i$ and $j$; as any pair of successively adjacent flags belongs  to the same orbit, the same must hold for the two given flags. Note that these flags may or may not belong to the same orbit as $\Phi$. This gives two possibilities. In either case we must show that $\Sigma\neq\emptyset$. 

Now, if the flags through $\Phi_i\phi$ and $\Phi_j\psi$ are in the same orbit as $\Phi$, and $\Psi$ is any such flag, then $\Psi=\Phi\alpha$ for some $\alpha \in \Gamma$. It follows that $\Phi_i\phi =  \Phi_i\alpha$, $\Phi_j\psi=\Phi_j\alpha$, and therefore $\alpha \in \Gamma_i\phi \cap \Gamma_j \psi$. Thus, $\alpha\in\Sigma$ and $\Sigma\neq \emptyset$.

On the other hand, if the flags through $\Phi_i\phi$ and $\Phi_j\psi$ are not in the same orbit as $\Phi$, we can argue as follows. Let $\Psi$ be any flag through $\Phi_i\phi$ and $\Phi_j\psi$. Then, since $j \notin I$, the flag $\Psi$ must be in the same orbit as $\Phi^j$, so there exists an element $\alpha \in \Gamma$ such that $\Psi=\Phi^j \alpha$. But then 
\[\Phi_{i}\phi=\Psi_{i}=(\Phi^j\alpha)_{i}=(\Phi^j)_{i}\,\alpha=\Phi_{i}\alpha\]
and
\[\Phi_{j}\psi =\Psi_{j} = (\Phi^{j}\alpha)_{j} =(\Phi^{j})_{j}\, \alpha = (\Phi^{j, i})_{j}\, \alpha =(\Phi \alpha_{j,i})_{j}\, \alpha =\Phi_{j}\,\alpha_{j,i} \,\alpha ,\]
and therefore $\alpha \in \Gamma_i\phi$ and $\alpha \in \alpha_{j,i}^{-1}\Gamma_{j}\psi = \alpha_{i,j}\Gamma_{j}\psi$, giving $\alpha\in\Gamma_i\phi \cap \alpha_{i,j}\Gamma_j \psi$. Thus, as in the previous case, $\alpha\in\Sigma$ and $\Sigma\neq \emptyset$.
\end{proof}
\medskip

It remains to complete the characterization of incidence of faces when ${\rm rd}(I)=1$, with $\bar{I}=\{j_0\}$ (say). Even though the polytope is not fully-transitive in this case, most instances for incidence of pairs of faces are still covered by Lemma~\ref{char1} and only involve the stabilizers of faces in the base flag~$\Phi$. However, in the remaining instances, namely when one face is a $j_0$-face, the characterization also involves the stabilizer of the $j_0$-face in the $j_0$-adjacent flag $\Phi^{j_0}$ of $\Phi$. More explicitly, when $\bar{I}=\{j_0\}$ the automorphism group $\Gamma=\Gamma(\mathcal{P})$ has a single orbit on the $l$-faces of $\mathcal{P}$ for each $l$ with $l \neq j_0$, but has two orbits on the $j_0$-faces of $\mathcal{P}$. The two orbits on the $j_0$-faces are represented by the $j_0$-face $\Phi_{j_0}$ of $\Phi$ and the $j_0$-face $\Phi^{j_0}_{j_0}$ of $\Phi^{j_0}$, whose stabilizers in $\Gamma$ are $\Gamma_{j_{0}}$ and~$\Gamma'_{j_{0}}$, respectively. Recall from Lemma~\ref{gamprimestabs} that $\Gamma'_{l}$ is the stabilizer of the $l$-face $\Phi_{l}^{j_0}$ in $\Phi^{j_0}$ for each $l$, and that $\Gamma'_{l}$ is generated as described in (\ref{gamsublprime}). 
\smallskip

The following lemma, in conjunction with Lemmas~\ref{char1} and \ref{char2}, settles the case when ${\rm rd}(I)=1$, except in one trivial instance discussed in (\ref{equiv_n-1Aplus}) below.

\begin{lemma}
\label{incnminus1}
Let $\mathcal{P}$ be a two-orbit $n$-polytope in the class $2_I$, with ${\rm rd}(I)=1$, and let~$\Phi$ be the base flag of $\mathcal{P}$. Let $i,j$ be such that $-1 \leq i < j \leq n$ and $\bar{I}=\{i\}$ or $\bar{I} = \{j\}$, and let $\phi,\psi \in \Gamma(\mathcal{P})$. \\[.04in]
(a)\ If $\bar{I} = \{i\}$, then the following equivalence holds: 
\begin{eqnarray}
\label{equiv_n-1A}
\Phi^i_i\,\phi\, \leq \Phi_j\,\psi  \,\Longleftrightarrow\, \Gamma'_i\,\phi \,\cap \,\Gamma_j \psi \neq \emptyset.
\end{eqnarray}
(b)\ If $\bar{I} = \{j\}$, then the following equivalence holds: 
\begin{eqnarray}
\label{equiv_n-1B}
\Phi_i\,\phi\, \leq \Phi^{j}_{j}\,\psi \,\Longleftrightarrow\, \Gamma_i\,\phi \,\cap\, \Gamma'_{j} \psi \neq \emptyset.
\end{eqnarray}
\end{lemma}

\begin{proof}
We only prove the first part. The second part can be verified similarly.

So, let $\bar{I} = \{i\}$. First, suppose that $\Phi^i_i\,\phi\, \leq \Phi_j\,\psi$. Then any two flags containing the $i$-face $\Phi^i_i\phi$ lie in the same orbit; once again, this follows from Lemma~\ref{rhoex} and the definition of the class $2_I$. But $\Phi_i^i\phi$ is the $i$-face in the flag~$\Phi^i\phi$, and therefore any flag containing $\Phi_i^i\phi$ is equivalent under $\Gamma$ to $\Phi^i\phi$ and thus to $\Phi^i$. We can now argue as before. If $\Psi$ is any flag containing $\Phi^i_i\,\phi$ and $\Phi_j\,\psi$, then there exists an element $\alpha \in \Gamma$ such that $\Psi=\Phi^i \alpha$. Hence 
\[\Phi^i_i\,\phi =\Psi_{i}= \Phi^i_i \alpha\] 
and 
\[\Phi_j\,\psi =\Psi_{j} = \Phi_j^i \alpha = \Phi_j \alpha,\] 
and therefore $\alpha \in  \Gamma'_i\,\phi \,\cap \,\Gamma_j \psi$. Thus $\Gamma'_i\,\phi \,\cap \,\Gamma_j \psi\neq\emptyset$. 

The converse is straightforward. If $\Gamma'_i\,\phi \,\cap \,\Gamma_j \psi\neq\emptyset$ and $\alpha \in  \Gamma'_i\,\phi \,\cap \,\Gamma_j \psi$, then 
\[\Phi^i_i\,\phi =\Phi^i_i \alpha\leq \Phi_j^i \alpha=\Phi_j \alpha=\Phi_j\,\psi,\]
as required.
\end{proof}
\smallskip

Note that the case $j=i$ was excluded in Lemma~\ref{incnminus1}, for the reason that the two $i$-faces $\Phi_i$ and $\Phi^i_i$ of $\mathcal{P}$ lie in distinct orbits when $\bar{I}=\{i\}$.
\smallskip

Finally, to complete the discussion of equality of faces when $\bar{I}=\{i\}$, we only need to observe that $\Gamma'_i$ is the stabilizer of $\Phi^i_i$ in this case and therefore
\begin{eqnarray}
\label{equiv_n-1Aplus}
\Phi^i_i\,\phi=\Phi^i_i\,\psi \,\Longleftrightarrow\, \Gamma'_i\,\phi \,\cap \,\Gamma'_i \psi \neq \emptyset.
\end{eqnarray} 
The equality of $i$-faces of $\mathcal{P}$ of the form $\Phi_i\varphi$, as well as the equality of $j$-faces with $j\neq i$, are covered by Lemma~\ref{char1}.
\medskip

Summarizing the discussion in this section, we have established the following theorem.

\begin{theorem}
\label{ordercharacterize}
Let $\mathcal{P}$ be a two-orbit $n$-polytope in the class $2_I$, with $I\subset N$, and let~$\Phi$ be the base flag of $\mathcal{P}$. Then the partial order on $\mathcal{P}$ can be characterized in terms of the collections of distinguished subgroups $\big\{\Gamma_l\big\}_{l\in N}$ and $\big\{\Gamma'_l\big\}_{l\in N}$ of the automorphism group $\Gamma(\mathcal{P})$ determined by $\Phi$ and certain flags adjacent to $\Phi$. In particular, depending on the class type set $I$ and the face ranks involved, incidence between a pair of faces of $\mathcal{P}$ is described by the equivalences in (\ref{equivnminus3}), (\ref{equivalence_n-2}), (\ref{equiv_n-1A}), (\ref{equiv_n-1B}) or (\ref{equiv_n-1Aplus}).  
\end{theorem}

\section{Small reflection deficiency}
\label{srd}

Recall that, for a two-orbit $n$-polytope $\mathcal{P}$ in the class $2_I$ with $I\subset N=\{0,\ldots,n-1\}$, the cardinality of the complement of $I$ is called the {\em reflection deficiency\/} of $\mathcal{P}$ and is denoted ${\rm rd}(\mathcal{P})$. Thus, ${\rm rd}(\mathcal{P})={\rm rd}(I)=|\bar{I}|$.

In the case of reflection deficiency 1, and in special cases of reflection deficiency 2, the group $\Gamma(P)$ of a two-orbit $n$-polytope $\mathcal{P}$ has additional intersection properties which generally do not seem to be implied by the standard intersection properties of Theorem~\ref{intcondlemma}. To discuss these properties let again~$\mathcal{P}$ belong to the class $2_I$, $I\subset N$, and set $\Gamma:=\Gamma(\mathcal{P})$. 

We begin with polytopes $\mathcal{P}$ of reflection deficiency 1. So, $\bar{I}=\{j_0\}$ (say). Then no non-trivial generators $\alpha_{j,k}$ with $j,k\notin I$ occur, and the set of distinguished generators reduces to a smaller set consisting only of the elements $\rho_i$ and $\alpha_{j_{0},i,j_{0}}$ with $i\in I$. Thus, 
\begin{equation}
\label{gengamsmall}
\Gamma = \langle \rho_i,\alpha_{j_{0},i,j_{0}} \mid i \in I\rangle.
\end{equation}
The additional intersection properties then take the following form.

\begin{lemma}
\label{intrefdef}
Let $\mathcal{P}$ be a two-orbit $n$-polytope in the class $2_I$, with $|I|=n-1$ and $\bar{I}=\{j_0\}$. Then, 
\begin{equation}
\label{intrefdef1one}
\begin{array}{cl}
\Gamma_{N\setminus\{j_{0}-1,j_{0}\}} \cap (\Gamma'_{j_0})^{-} = 
\langle\alpha_{j_{0},j_{0}-1,j_{0}}\rangle
&\;(j_{0}\neq 0),\\[.05in]
\Gamma_{N\setminus\{j_{0},j_{0}+1\}} \cap (\Gamma'_{j_0})^+ = 
\langle\alpha_{j_{0},j_{0}+1,j_{0}}\rangle
&\;(j_{0}\neq n-1), 
\end{array}
\end{equation}
or equivalently, expressed in terms of the generators, 
\begin{equation}
\label{intrefdef1two}
\begin{array}{l}
\langle\rho_{{j_0}-1},\alpha_{j_{0},j_{0}-1,j_{0}}\rangle 
\,\cap\, \langle\alpha_{j_{0},i,j_{0}}\mid i<j_{0}\rangle
\,=\, \langle\alpha_{j_{0},j_{0}-1,j_{0}}\rangle
\quad\, (j_{0}\neq 0),\\[.05in]
\langle\rho_{{j_0}+1},\alpha_{j_{0},j_{0}+1,j_{0}}\rangle 
\,\cap\, \langle\alpha_{j_{0},i,j_{0}}\mid i>j_{0}\rangle
\,=\, \langle\alpha_{j_{0},j_{0}+1,j_{0}}\rangle
\quad\, (j_{0}\neq n-1).
\end{array}
\end{equation}
\end{lemma}

\begin{proof}
The proof employs stabilizers of faces or chains closely related to the base flag $\Phi$ of~$\mathcal{P}$. 

We begin with the proof of the first statement. By Lemma~\ref{gamstab}, $\Gamma_{N\setminus\{j_{0}-1,j_{0}\}}$ is the stabilizer of the subchain $\Phi_{N\setminus\{j_{0}-1,j_{0}\}}$ of $\Phi$; and by Lemma~\ref{gamprimestabs}, the subgroup $(\Gamma'_{j_0})^{-}$ of $\Gamma'_{j_0}$ stabilizes the $j_{0}$-face $(\Phi^{j_0})_{j_0}$ of the $j_0$-adjacent flag $\Phi^{j_0}$ of $\Phi$. It follows that the intersection on the left hand side of the first equation in (\ref{intrefdef1one}) stabilizes the chain $\Phi_{N\setminus\{j_{0}-1,j_{0}\}}\cup\{(\Phi^{j_0})_{j_0}\}$ of $\mathcal{P}$ and therefore has order at most 2. On the other hand, the involution $\alpha_{j_{0},j_{0}-1,j_{0}}$ occurring on the right hand side of the first equation in (\ref{intrefdef1one}) trivially lies in this  intersection. Therefore the groups on the left and right hand sides coincide and their  order is 2.

The proof of the second statement of (\ref{intrefdef1one}) follows the same pattern. Now the group on the left side in the second equation must stabilize the chain $\Phi_{N\setminus\{j_{0},j_{0}+1\}}\cup\{(\Phi^{j_0})_{j_0}\}$ and hence must coincide with the group generated by $\alpha_{j_{0},j_{0}+1,j_{0}}$ on the right side.
\end{proof}
\smallskip

Next suppose a two-orbit $n$-polytope $\mathcal{P}$ has reflection deficiency 2 and lies in $2_I$, with $\bar{I}=\{j_0,k_0\}$ and $k_{0}=j_{0}+2$. Note the restrictive assumption on the pair $j_0,k_0$. Then we have the following additional intersection properties.

\begin{lemma}
\label{intrefdefnew}
Let $\mathcal{P}$ be a two-orbit $n$-polytope in the class $2_I$, with ${\rm rd}(I)=2$, $\bar{I}=\{j_0,k_0\}$ and $k_{0}=j_{0}+2$. Then, 
\begin{equation}
\label{intrefdef2}
\begin{array}{rl}
\Gamma_{k_0}^{-}\,\cap\,\Gamma_{j_0}^{+}\alpha_{j_0,k_0}\!\!\! &=\; \emptyset,\\[.08in]
\Gamma_{k_0}^{-}\,\cap\,\alpha_{k_0,j_0}\Gamma_{j_0}^{+}\alpha_{j_0,k_0} \!\!\!&=\, 
\langle \alpha_{j_0,j_{0}+1,j_{0}}\rangle.\\[.05in]
\end{array}
\end{equation}
\end{lemma}

\begin{proof}
First note that $\alpha_{j_0,k_0}$ is an involution since $k_{0}=j_{0}+2$. The proofs will again employ stabilizers of faces or chains closely related to the base flag $\Phi$ of~$\mathcal{P}$. 

Beginning with the first statement, suppose to the contrary that 
$\Gamma_{k_0}^{-}\,\cap\,\Gamma_{j_0}^{+}\alpha_{j_0,k_0}\neq\emptyset$ 
and let $\gamma\in\Gamma_{k_0}^{-}\cap\Gamma_{j_0}^{+}\alpha_{j_0,k_0}$. Then, by Lemma~\ref{gamstab}, $\gamma$~stabilizes each face $\Phi_l$ of $\Phi$ with $l\geq k_0$, and $\Phi_{l}\gamma=\Phi_{l}\alpha_{j_0,k_0}$ for each $l\leq j_0$ since $\Gamma_{j_0}^{+}$ stabilizes each $\Phi_l$ with $l\leq j_0$. By the definition of~$\alpha_{j_0,k_0}$,  
\[\Phi_{l}\alpha_{j_0,k_0} = (\Phi\alpha_{j_0,k_0})_{l} = (\Phi^{j_0,k_0})_l\]
for each $l\in N$, and therefore $\Phi_{l}\alpha_{j_0,k_0}=\Phi_l$ for $l\neq j_0,k_0$. Moreover, 
\[\Phi_{j_0}\gamma=\Phi_{j_0}\alpha_{j_0,k_0}=(\Phi\alpha_{j_0,k_0})_{j_0}
=(\Phi^{j_0,k_0})_{j_0}=(\Phi^{j_0})_{j_0}.\]
It follows that $\gamma$ fixes the entire chain $\{\Phi_l\mid l\neq j_0,j_{0}\!+\!1\}$ and thus acts faithfully on the polygonal section $\Phi_{k_0}/\Phi_{j_{0}-1}$ of rank 2. On this section, $\gamma$ acts as an automorphism that maps the base vertex $\Phi_{j_0}$ to the adjacent vertex $(\Phi^{j_0})_{j_0}$ in the base edge $\Phi_{j_{0}+1}$ of $\Phi_{k_0}/\Phi_{j_{0}-1}$. But since $j_0\notin I$, the polytope $\mathcal{P}$ does not admit an automorphism that maps $\Phi$ to $\Phi^{j_0}$, so $\gamma$ cannot interchange $\Phi_{j_0}$ and $(\Phi^{j_0})_{j_0}$ and must necessarily permute the vertices in $\Phi_{k_0}/\Phi_{j_{0}-1}$ cyclically by one step. On the other, since $j_{0}+1\in I$, the polytope $\mathcal{P}$ does admit an automorphism, namely $\rho_{j_{0}+1}$, which interchanges the base edge $\Phi_{j_{0}+1}$ of $\Phi_{k_0}/\Phi_{j_{0}-1}$ with the adjacent edge at vertex $\Phi_{j_0}$. Hence $\langle\gamma,\rho_{j_{0}+1}\rangle$ must be the full dihedral automorphism group of $\Phi_{k_0}/\Phi_{j_{0}-1}$ and must necessarily contain an element that interchanges the vertices in the base edge $\Phi_{j_{0}+1}$ of $\Phi_{k_0}/\Phi_{j_{0}-1}$. It follows that $\mathcal{P}$ must admit an automorphism that maps the base flag $\Phi$ to its $j_0$-adjacent flag $\Phi^{j_0}$, which is a contradiction to our assumption that $j_0\notin I$. Therefore, 
$\Gamma_{k_0}^{-}\cap\Gamma_{j_0}^{+}\alpha_{j_0,k_0}=\emptyset$. 

The second statement can be derived as follows. First observe that the element $\alpha_{j_0,j_{0}+1,j_{0}}$ occurring on the right hand side of the equation lies in the intersection on the left hand side. In fact, $\alpha_{j_0,j_{0}+1,j_{0}}\in\Gamma_{k_0}^{-}$  and by (\ref{rel0more}),
\[\alpha_{j_0,j_{0}+1,j_{0}} = \alpha_{k_0,j_0}\alpha_{k_0,j_{0}+1,k_{0}}\alpha_{j_0,k_0},\]
with $\alpha_{k_0,j_{0}+1,k_{0}}\in\Gamma_{j_0}^{+}$. Thus, 
$\alpha_{j_0,j_{0}+1,j_{0}}\in\Gamma_{k_0}^{-}\cap\alpha_{k_0,j_0}\Gamma_{j_0}^{+}\alpha_{j_0,k_0}$. 
This establishes one inclusion.

For the proof of the opposite inclusion, let  
$\gamma\in\Gamma_{k_0}^{-}\cap\alpha_{k_0,j_0}\Gamma_{j_0}^{+}\alpha_{j_0,k_0}$. 
Then $\gamma$ stabilizes each face $\Phi_l$ with $l\geq k_0$, and $\alpha_{k_0,j_0}^{-1}\gamma\alpha_{j_0,k_0}^{-1}$ stabilizes each face $\Phi_l$ with $l\leq j_0$. The latter shows that $\gamma$ stabilizes each face $\Phi_l\alpha_{j_0,k_0}$ with $l\leq j_0$. But $\Phi_l\alpha_{j_0,k_0}$ is just $\Phi_l$ if $l<j_0$, and coincides with $(\Phi^{j_0})_{j_0}$ if $l=j_0$. It follows that $\gamma$ must stabilize the chain 
\[\{\Phi_l\mid l\neq j_0,j_{0}\!+\!1\}\cup\{(\Phi^{j_0})_{j_0}\},\]
which in turn is just $\Phi^{j_0}\setminus\{(\Phi^{j_0})_{j_{0}+1}\}$. Thus $\gamma$ either fixes the flag $\Phi^{j_0}$ or interchanges the flags $\Phi^{j_0}$ and $(\Phi^{j_0})^{j_{0}+1}$. Clearly, if $\gamma$ fixes $\Phi^{j_0}$, then $\gamma=1$. However, if $\gamma$ interchanges $\Phi^{j_0}$ and $(\Phi^{j_0})^{j_{0}+1}$,
then 
\[\Phi\gamma = (\Phi^{j_0})^{j_0}\gamma = (\Phi^{j_0}\gamma)^{j_0}
= ((\Phi^{j_0})^{j_{0}+1})^{j_0}=\Phi^{j_0,j_{0}+1,j_0} = \Phi\alpha_{j_0,j_{0}+1,j_0}\]
and therefore $\gamma=\alpha_{j_0,j_{0}+1,j_0}$. Either way, $\gamma\in\langle\alpha_{j_0,j_{0}+1,j_0}\rangle$, as required.
\end{proof}
\smallskip

\subsection*{After-note}
As mentioned in the Introduction, the present article has existed in nearly complete form as an unpublished preprint for several years.
Over this time, our results have inspired a number of developments on symmetric abstract polytopes and maniplexes with two or more flag orbits under the automorphism group; we conclude this paper by discussing some of these developments and lines of work.

To start, we point out that our study has inspired the study of different classes of two-orbit polytopes as well as that of some $k$-orbit polytopes, such as those that one can find in \cite{CP2018, HM2022, MiSW2014, MoS2012, MoS2021, MoS2022}.

One prominent line of investigation that has led to major developments on abstract polytopes is that of their geometric realizations in Euclidean spaces, often referred to as skeletal polytopes. 
The articles by Gr\"unbaum~\cite{G1977a} and Dress~\cite{D1981,D1985} on the classification of geometrically regular polyhedra in ordinary space, now often called Gr\"unbaum-Dress polyhedra, provided an important impetus to this line of inquiry (see also \cite{MS1997}). 
{The recent Geometric Regular Polytopes monograph by McMullen~\cite{Mc2020}offers a comprehensive account on geometric realizations of abstract regular polytopes in Euclidean spaces.} Significant progress has { also} been made in the classification of skeletal chiral polytopes:\ 
those of rank $3$ in ordinary space were classified  { in Schulte~\cite{S2004,S2005}, while Pellicer~\cite{P2017} enumerated those of rank~$4$ in ordinary space}. Results for skeletal chiral polytopes in higher dimensional Euclidean spaces can be found in \cite{BG-CH2027, BHP2014, BHP2016, BHP2021, HT2026+, Pe2016, P2017, P2021, Pe2021, P2025, PWe2010}.
 As for other classes of two-orbit skeletal polytopes as well as skeletal polyhedra with several flag orbits, we refer the reader to \cite{BHP2016, CP2023, CP2024, HGV,    Ma2016a, Ma2016b, Mc2021, P2016, PWi2023, PWi2026+, SW2020}. 
These results reveal an interplay between combinatorial and geometric symmetry: geometric two-orbit polyhedra may be combinatorially regular, and conversely, combinatorial two-orbit behavior may arise from distinct geometric realizations.

A unifying perspective for studying polytopes with $k$ flag orbit is provided by the notion of a symmetry type graph, which encodes the orbit structure of flags under the automorphism group. 
In the two-orbit case, this graph has two vertices {and its edges are} colored by the ranks of the polytope, indicating how adjacent flags relate across or within orbits. 
While this framework suggests a wide range of possible symmetry types, a central problem is to determine which of these types are actually realized by abstract polytopes.
Additionally, the notion of an abstract polytope itself has been generalized to that of maniplexes (\cite{W2012}), a relatively new concept that bridges abstract polytopes and maps on surfaces,
and includes a much broader class of objects exhibiting weaker connectedness than abstract polytopes. { Polytopality is an important theme in maniplexes:\ 
necessary and sufficient conditions for a maniplex} to be the flag graph of a polytope have been described in~\cite{GVH2018}. {Symmetry type graphs, as well as voltage assignments of graphs, have greatly influenced the study of maniplexes.} 
Several publications on two-orbit or $k$-orbit structures { have exploited this connection 
(see for example ~\cite{CRHT2015,CP2018,HM2022,HM2023, Mo2021,Mo2024}).
Additionally, two-orbit hypermaps and chiral hypertopes have been investigated in the literature~\cite{ruiphd,FLW2016}.}



The reader may be wondering why we held off on publishing this article at an earlier time. The original plan for our work on two-orbit polytopes included a companion preprint, either ultimately integrated into the present article or to be published separately. This companion preprint was meant to fully characterize the groups that occur as automorphism groups of two-orbit polytopes of any rank and class. Despite significant efforts and progress, and largely also due to other demands on our time, we never completed the companion preprint. 
In the meantime, a characterization of groups of two-orbit polytopes has been obtained in Hubard \& Mochan~\cite{HM2023} using different generating sets for the groups (that can be derived from { those} given here, and vice-versa). 
As the present article fully stands on its own feet, we believe that it is overdue to publish our findings. 

\section*{Acknowledgments}

The first author was partially supported by Secihti-M\'exico under grant CBF-2025-I-224, and UNAM-PAPIIT IN108926. The work of the second author was partially supported by Simons Foundation Award No.\,420718. We also wish to thank Asia Weiss for numerous discussions on two-orbit polytopes and her continuing encouragement to complete this paper.

\end{document}